\documentclass[10pt]{amsart}
\usepackage{amsmath, amscd, amsthm} 
\usepackage{amssymb, amsfonts, enumitem} 
\usepackage{graphicx}
\usepackage[all]{xy} 
\usepackage{aliascnt}
\usepackage[svgnames]{xcolor} 
\usepackage{hyperref}
  \definecolor{dark-red}{rgb}{0.4,0.15,0.15}
  \hypersetup{
    colorlinks, linkcolor=dark-red,
    citecolor=DarkBlue, urlcolor=MediumBlue
}  
\newcommand{\C}{\mathbb{C}} 
\newcommand{\A}{\mathbb{A}}
\renewcommand{\P}{\mathbb{P}}
\newcommand{\Z}{\mathbb{Z}}
\newcommand{\R}{\mathbb{R}}

\newcommand{\iso}{\cong}

\newcommand{\rH}{\widetilde{H}}

\newcommand{\EG}{\mathbf{E}}
\newcommand{\BG}{\mathbf{B}}
\newcommand{\EGt}{\widetilde{\EG} }

\newcommand{\mcal}[1]{\mathcal{#1}}
\newcommand{\ul}[1]{\underline{\smash{#1}}}
\newcommand{\MZ}{\mathbf{M}\ul{\Z}}
\newcommand{\MA}{\mathbf{M}\ul{A}}

\numberwithin{equation}{section} 

\theoremstyle{plain} 
 
\newaliascnt{theorem}{equation}  
\newtheorem{theorem}[theorem]{Theorem}  
\aliascntresetthe{theorem}

\newaliascnt{proposition}{equation}  
\newtheorem{proposition}[proposition]{Proposition}
\aliascntresetthe{proposition}

\newaliascnt{lemma}{equation}    

\aliascntresetthe{lemma}

\newaliascnt{corollary}{equation}  
\newtheorem{corollary}[corollary]{Corollary}
\aliascntresetthe{corollary}

\newaliascnt{claim}{equation}  

\aliascntresetthe{claim}

 \theoremstyle{definition}

\newaliascnt{definition}{equation}  
\newtheorem{definition}[definition]{Definition}
\aliascntresetthe{definition}

\newaliascnt{example}{equation}  

\aliascntresetthe{example}

\newaliascnt{remark}{equation}   

\aliascntresetthe{remark}

\newaliascnt{condition}{equation}

\aliascntresetthe{condition}

\newaliascnt{notationconvention}{equation}

\aliascntresetthe{notationconvention}

\author{Mircea Voineagu}
\email{m.voineagu@unsw.edu.au}
\address{Department of Mathematics and Statistics, The Red Centre, Kensington Campus, UNSW Sydney, NSW 2052 Australia}

\begin{document}
\title{About Bredon motivic cohomology of a field}
 
  \maketitle
  \begin{abstract}
We prove that, over a perfect field, Bredon motivic cohomology can be computed by Suslin-Friedlander complexes of  equivariant equidimensional cycles. Partly based on this result we completely identify Bredon motivic cohomology of a quadratically closed field and of a euclidian field in weights 1 and $\sigma$.  We also prove that Bredon motivic cohomology of an arbitrary  field  in weight 0 with integer coefficients coincides (as abstract groups) with Bredon cohomology of a point. 
\end{abstract} 
  \tableofcontents
 \section{Introduction}
  Motivated by Voevodsky's definition of motivic cohomology, we introduced in \cite{HOV} (together with J.Heller and P.A. Ostvaer) an equivariant generalization of motivic cohomology called Bredon motivic cohomology for $G-$ equivariant smooth schemes over a field $k$ and for a finite group $G$. In a subsequent  paper \cite{HOV1} (together with J.Heller and P.A. Ostvaer), we extended this definition to an equivariant motivic cohomology bigraded by a pair of virtual $C _2-$representations and defined in the $\Z/2$-equivariant stable $\A^1-$ homotopy category. 
  
  We proved in \cite{HOV1} that for a  smooth $C _2-$ scheme $X$ over the field of complex numbers there is a natural cycle map into Bredon cohomology of $X(\C)$, which in a certain range of indexes becomes an isomorphism with finite coefficients. This result allows us to obtain partial computations of Bredon motivic cohomology groups of a complex variety  in terms of Bredon cohomology of its complex points. However, apart from applications of this theorem, there is a lack of computations of Bredon motivic cohomology groups.
  
  In this paper we give complete computations of Bredon motivic cohomology groups with integer coefficients in weight $0$ of an arbitrary field. We also compute Bredon motivic cohomology of a field of characteristic zero with integer coefficients of bi-index given by one-dimensional $\Z/2$-representations. We also give complete computations  for quadratically closed fields and euclidian fields of characteristic zero in weights  $1$ and $\sigma$. The strategy to prove the results in weights $1$ and $\sigma$ (see section 5) is to show that Bredon motivic cohomology can be computed by Suslin-Friedlander complexes of equivariant equidimensional cycles (see section 3). This together with a result of Nie \cite{Nie} gives a computation of the Bredon cohomology of a field of characteristic zero in bidegree $(a+2q\sigma,b+q\sigma)$ (see Proposition \ref{Nie}). Using this computation and after a careful analysis of the maps that appear in the l.e.s. given by the cofiber sequence $C _{2+}\rightarrow S^0\rightarrow S^\sigma$ we complete the computation of Bredon motivic cohomology with integer coefficients in weights 1 and $\sigma$ of certain fields (see section 5 Theorems \ref{com}, \ref{sigm}, \ref{neg1}, \ref{pos1}). 
  
  A consequence of this computation is that Bredon motivic cohomology with $\Z/2$ coefficients of a quadratically closed field of characteristic zero in weights 1 or $\sigma$ coincides (as abstract groups) with Bredon cohomology of a point with $\Z/2-$coefficients (but they have different computations with integer coefficients). For this see Corollary \ref{sif} and Corollary \ref{fcoe}. 
  
  For example, in weight 1 with $\Z/2-$coefficients we have the following (see Corollary \ref{fcoe}): 
  \begin{theorem} Let $k$ be a field of characteristic zero. If $k$ is a quadratically closed field or a euclidian field we have   
\begin{equation}
H^{a+p\sigma,1}_{C _2}(k,\Z/2)= \left\{
\begin{array}{ll}
       \Z/2\oplus k^*/k^{*2} & 0\leq -a\leq p-1\\
      \Z/2\oplus k^*/k^{*2} & 2<a\leq -p\\
      \Z/2 & a=2\leq -p, 0\leq -a=p\\
      k^*/k^{*2} & a=1<p, 2<a=-p+1\\
    0 & otherwise\\
\end{array} 
\right. 
\end{equation}

Notice that in the case of a quadratically closed field (for example any algebraically closed field such as $\C$) $k^*/k^{*2}=0$ and in the case of a euclidian field (for example $\R$) $k^*/k^{*2}\simeq \Z/2$.
  \end{theorem}

  We know that  $H^{1,1}(k,\Z)=K^1 _M(k)=k^*$ (\cite{MVW}, \cite{Tot}). As a generalization of this result we have the following result for Bredon motivic cohomology (Corollary \ref{gh}):
  \begin{theorem} For an arbitrary  field $k$ of characteristic zero we have  $H^{\sigma,\sigma}_{C _2}(k,\Z)=\Z/2$, $H^{1,\sigma}_{C _2}(k,\Z)=k^{*2}$, $H^{\sigma,1}_{C _2}(k,\Z)=\Z/2$ and $H^{1,1}_{C _2}(k,\Z)=k^*$. 
\end{theorem}
  
  In section $4$ we analyze the cohomology groups of the shift complexes $Z _{top}(n\sigma)(k)$ and $Z _{top}(n)(k)$ for an arbitrary field $k$. We obtain  computations for weight 0 Bredon motivic cohomology with integer coefficients of a field. 
  
 As a consequence, we obtain that Bredon motivic cohomology with integer coefficients  in weight $0$ of an arbitrary field coincides (as abstract groups) with Bredon cohomology of a point with integer coefficients. In conclusion, with $\Z/2$-coefficients we have:
  \begin{theorem} For an arbitrary field $k$ we have 
\begin{equation}
H^{a+p\sigma,0} _{C _2}(k,\Z/2)= \left\{
\begin{array}{ll}
      \Z/2 & 0\leq -a \leq p \\
      \Z/2 & 1<a\leq -p\\
    0 & otherwise\\
\end{array} 
\right. 
\end{equation}
  \end{theorem}
  In section $3$ we show that Suslin-Friedlander complexes of $C _2-$equivariant equidimensional cycles compute Bredon motivic cohomology over a perfect field (see Theorem \ref{id}). 
  In section $6$ we discuss Borel motivic cohomology of a field in weights 0,1 and $\sigma$. 
  
  {\bf{Notation and Conventions.}} We let $G=\Z/2$ and we view it as a group scheme over $k$ via $G=\sqcup {Spec(k)}$. Sometimes we write $C _2$ for the same group. We let $GSch/k$ be the category of $G-$equivariant separated quasi-projective schemes  with $G$-equivariant morphisms of schemes and $GSm/k$  be its subset of smooth G-equivariant quasi-projective schemes over $k$. 
 
 We write $A(V)=Spec(Sym(V^\vee))$ for the affine k-scheme associated to a k-vector space $V$ and $\P(V)=Proj(A(V))$ for the projective scheme.
 
  We denote $S _x=\{g\in G|gx=x\}=G/Gx$ the set-theoretical stabilizer of a point $x$ and $Gx=G\times^{S _x}\{x\}$ the orbit of a point $x$. In general for a $Z\subset X$ we write $GZ=\cup _{g\in G} gZ$ for the orbit of $Z$. 
  
  For an abelian group $G$ we write $_2G:=\{g\in G| 2g=0\}$. We write $k[G]=1\oplus \sigma$ for the regular $C_2-$representation. We denote by $\sigma$ the real $C _2$-representation given by $\R$ with the $-1$ action. We denote by $S^\sigma$ the sphere associated to the $C _2-$representation $\sigma$ and in general by $S^V$ the topological sphere associated to the $C _2$-representation $V$. 
  
  We use the following notation for indexes in $RO(G)$-graded Bredon cohomology $H^{p,q} _{Br}(X,\Z):=H^{p+q\sigma} _{Br}(X,\Z)$ for any virtual $C _2-$ representation $V=p+q\sigma$. We write $H^{n,q}(X)$ for the motivic cohomology of the scheme $X$.
 
 {\bf{Acknowledgements:}} The author would like to thank Jeremiah Heller for many useful discussions. The author would also like to thank the University of Illinois at Urbana-Champaign, where the final stages of this project were completed. The author thanks the referee for many useful comments.
 \section{Preliminaries}
  Let $V$ be a $C _2$-representation and $$T^V:=\P(V\oplus 1)/\P(V)\simeq \A(V)/\A(V)\setminus\{0\}$$ be the motivic representation sphere with the isomorphism in the $C _2-$ equivariant $\A^1-$ homotopy category (\cite{Del}, \cite{HKO}, \cite{HOV1}). For $n\geq 0$ we write $T^{nV}={T^V}\wedge...\wedge T^V$ with $n$ copies of $T^V$ in the right term. More generally we have $T^{V\oplus W}=T^V\wedge T^W$ for any two $C _2$-representations $V$ and $W$.
   
   We have by definition that $$A(V):=a_{eNis} C _*A _{tr,C _2}(T^V)[-2a-2b\sigma]$$ is the Bredon motivic complex associated to the $C _2-$representation $V=a+b\sigma$. Here $A _{tr,C _2}(X)=Cor _k(-,X)^{C _2}\otimes A$ and $a _{eNis}$ is a sheafification in equivariant Nisnevich topology (\cite{Del}, \cite{HOV}, \cite{HOK}).  Here $Cor _k(X,Y)^{C _2}$ is the free abelian group generated by elementary equivariant correspondences. An equivariant elementary correspondence from $X$ to $Y$ is a correspondence of the form $\tilde {Z}=Z+g _1Z +g _2Z +...+g_nZ$ where $Z$ is an elementary correspondence (finite and surjective over a connected component of $X$) and $g _i$ range over a set of coset representatives for $Stab(Z)=\{g\in G| gZ=Z\}$.
  
   The equivariant Nisnevich topology is the smallest Grothendieck topology containing the elementary Nisnevich covers $\{U\rightarrow X,Y\rightarrow X\}$ associated to the equivariant distinguished square in $C_2Sch/k$
  $$
\xymatrix{
W\ar[r]\ar[d]& Y\ar[d] _{p} \\
U \ar@{^{(}->}[r] _i & X.
} 
$$
Here $p:Y\rightarrow X$ is an equivariant etale morphism and $U\hookrightarrow X$ is an equivariant open embedding such that $(X\setminus U) _{red}\simeq (Y\setminus W) _{red}$. According to \cite{HOV} and \cite{HOK} an equivariant etale map $p:Y\rightarrow X$ is a $C _2-$ equivariant Nisnevich cover if for any $x\in X$ there is $y\in Y$ such that $p(y)=x$, $k(x)\simeq k(y)$ and the set-theoretical stabilizers coincide $S _x\simeq S _y$. According to \cite{HOV} the points in the equivariant Nisnevich topology are Hensel semi-local affine $C_2-$schemes with a single closed orbit. Any semilocal
Henselian affine $C_2-$scheme over $k$ with a single orbit is equivariantly isomorphic to $Spec(O^h _{A,Gx})$, the $Spec$ of the henselization of the semilocal ring $O _{A,Gx}$ where $A$ is some affine $C_2-$scheme and $x\in A$. 

For a different generalization of Nisnevich topology in the equivariant setting called the fixed point Nisnevich topology see \cite{H}. The fixed point Nisnevich topology is a topology finer than the equivariant Nisnevich topology, but coarser than the equivariant etale topology used by Thomasson \cite{Tho}. For example if $X$ is a $G-$scheme with free action and $X^{e}$ is the same scheme with trivial action then $X^{e}\times G\rightarrow X$ is a cover in fixed point topology, but not in the equivariant Nisnevich topology \cite{HOV}. 

 The shift in the definition of Bredon motivic complexes is given by the tensor product in  $D^-(C _2Cor _k)$ with the invertible complexes $Z _{top}(a+b\sigma)$ associated to any virtual $C _2-$representation $V=a+b\sigma$ \cite{HOV}. We have that $$Z _{top}(\sigma):=Cone(\Z _{tr}(C _2)^{C _2}\rightarrow \Z _{tr}(k)^{C _2})\in D^-(C _2Cor _k)$$ is the complex associated to the sign sphere $S^\sigma$. 
 Here $D^-(C _2Cor _k)$ is the derived category of bounded above chain complexes of equivariant Nisnevich sheaves with equivariant transfers on $C _2Sm/k$ and it has a tensor product induced by $$\Z _{tr}(X)^{C _2}\otimes^{tr} \Z _{tr}(Y)^{C _2}:=\Z _{tr}(X\times Y)^{C _2}.$$ 
   For an abelian group $A$ we have that  $\{\MA_n=\Z _{tr,C _2}(T^{nk[C _2]})\} _n$ defines the Bredon motivic cohomology spectrum $\MA$ in the stable $C _2$-equivariant motivic homotopy category $SH _{C _2}$ \cite{HOV1}. The construction of $SH _{C _2}$ as a stabilization with respect to the motivic sphere $T^{k[C _2]}$ is recalled in the appendix of \cite{HOV1} and it initially appeared in \cite{HKO} as a tool to study Hermitian K-theory of fields. We write $[-,-] _{SH _{C _2}}$ for maps in $SH _{C _2}$.
   
    The equivariant $\A^1-$homotopy was introduced initially by Voevodsky \cite{Del} in order to understand motivic Eilenberg-Mac Lane spaces. In the $C _2$-equivariant $\A^1$-homotopy category (see \cite{Del}, \cite{HKO}) we have the following spheres: the usual sphere $S^1$ with trivial action, the sign sphere $S^\sigma$ with the conjugation action, the Tate sphere $S^1 _t=(\A^1\setminus\{0\},1)$ with the trivial action and the sign Tate sphere $S^\sigma _t=(\A^1(\sigma)\setminus\{0\},1)$ with the action $x\rightarrow x^{-1}$. 
   
   Using the notation from \cite{HOV1} we define below the following motivic spheres bi-indexed by virtual $C _2$-representations: $$S^{a+p\sigma,b+q\sigma}:= S^{a-b}\wedge S^{(p-q)\sigma}\wedge S^{b} _{t}\wedge S^{q\sigma} _{t}.$$ Notice that this notation is slightly different from the one used in \cite{HKO} (see \cite{HOV1} for the translation between these two notations). This notion is suitable for comparison of Bredon motivic cohomology and Bredon cohomology as the complex realization of $Re(S^{a+p\sigma,b+q\sigma})=S^{a+p\sigma}$, where the right side defines the usual topological sphere associated to the $C _2-$ representation $V=a+p\sigma:=\R^a\oplus \R^{-p}$. 
   
   The definition of Bredon motivic cohomology in $SH _{C _2}$ is the following:
 \begin{definition} (see \cite{HOV}) \label{defB} The Bredon motivic cohomology of a motivic $C _2$-spectrum $E$ with coefficients in an abelian group $A$ is defined by 
 $$\rH^{a+p\sigma,b+q\sigma}(E,A):=[E, S^{a+p\sigma,b+q\sigma}\wedge M\underline{A}] _{SH _{C _2}}.$$
 We call the virtual $C _2$-representation $V=a+p\sigma$ the cohomology index and the virtual $C _2-$representation $W=b+q\sigma$ the weight index. 
 \end{definition}
 According to \cite{HOV} this definition for $E=\Sigma^\infty _{T^{k[G]}}X _+$, where $X$ is a smooth $C _2-$scheme and the motivic sphere $T^{k[G]}=T\wedge T^{\sigma}=S^1\wedge S^1 _t\wedge S^\sigma \wedge S^\sigma _t$ (\cite{HKO}, \cite{HOV1}), coincides with the definition of Bredon motivic cohomology given by the hypercohomology in equivariant Nisnevich topology (or equivariant Zariski topology \cite{HOV}) of the Bredon motivic complexes $A(V)$ defined above. 
 
 Based on an equivariant cancellation theorem proved in \cite{HOV} we can conclude that if $V=b+q\sigma$ is a virtual representation and $W=c+d\sigma$ is a representation such that $V\oplus W$ is an actual $C _2-$ representation then for any $C _2-$scheme $X$ we have
  $$H^{a+p\sigma,b+q\sigma}(X,\Z)= \mathbb{H}^{a} _{eNis}(X _+\wedge T^W, \Z(V\oplus W)[2c+(2d+p)\sigma]).$$

 For the relationship between the Bredon motivic cohomology and Edidin-Graham equivariant higher Chow groups \cite{EG} (which are a generalization of Totaro's  Chow groups of classifying spaces \cite{T}) see \cite{HOV1}.
 
 For a pointed motivic $C _2-$space $\chi$ we have that the  map in the equivariant $\A^1$-homotopy category $$C _{2+}\wedge \chi\simeq C _{2+}\wedge\chi^e\rightarrow \chi^e$$ induces an isomorphism $\rH^{a+p\sigma,b+q\sigma}_{C_2}(C _{2+}\wedge \chi,\Z)=\rH^{a+p,b+q}(\chi^{e},\Z)$, with $\chi^e$ being the motivic space with the forgotten action \cite{HOV1} and the right side is given by the usual motivic cohomology. If the $C _2$-scheme $X$ has trivial action then according to \cite{HOV1} we have that $H^{a,b}_{C_2}(X,\Z)=H^{a,b}(X,\Z)$. Proposition \ref{triv} in section $3$ gives a slight generalization of this last result. If $X$ is a smooth quasi-projective $C _2-$scheme
 with free $C _2$-action then $H^{n,m} _{C _2}(X,A)=H^{n,m}(X/C _2,A)$ \cite{HOV1}. 
 
 In the stable equivariant $\A^1-$ homotopy category there is the following motivic isotropy sequence 
 $$\EG\Z/2 _+\rightarrow S^0\rightarrow \EGt\Z/2$$
 defined as a colimit of $C _2-$homotopy cofiber sequences
 $$\A(n\sigma)\setminus\{0\} _+\rightarrow S^0\rightarrow S^{2n\sigma,n\sigma}=S^{n\sigma} _{t}\wedge S^n _{t}.$$
 By definition $\EG\Z/2=colim _n \A(n\sigma)\setminus\{0\}$. The first map is induced by the projection $\A(n\sigma)\setminus\{0\}\rightarrow Spec(k)$. We have that $\EGt\Z/2=colim _nS^{n\sigma} _{t}\wedge S^n _{t}$. The geometric classifying space is defined as the quotient of $\EG \Z/2$ by the free $C _2-$action $$\BG \Z/2:=\EG \Z _2/\Z _2=colim _n\A(n\sigma)\setminus\{0\}/\Z/2.$$ In Proposition 2.9 and Theorem 5.4 in \cite{HOV1} we showed that Bredon motivic cohomology of $\EG \Z/2$ is periodic of periodicity $(2\sigma-2,\sigma-1)$ and that Bredon motivic cohomology of $\EGt Z/2$ is periodic with periodicities $(\sigma-1,0)$ and $(\sigma-1,\sigma-1)$.

 In the stable equivariant $\A^1-$homotopy category we have the basic cofiber sequence
$$
C_{2\,+}\wedge \MZ \xrightarrow{p} \MZ \xrightarrow{i} S^{\sigma}\wedge \MZ
$$
which gives rise to the natural long exact sequence for any $C _2-$ scheme $X$
\begin{align} \label{connec}
\cdots \to 
H^{a-1+(p+1)\sigma, b+q\sigma}_{C_2}(X) & \xrightarrow{\delta} 
H^{a+p, b+q}_{\mcal{M}}(X) \xrightarrow{p_*} 
H^{a +p\sigma, b+q\sigma}_{C_2}(X) 
\\
& \xrightarrow{i_*} 
H^{a+(p+1)\sigma, b+q\sigma}_{C_2}(X) \xrightarrow{\delta} 
H^{a+p+1, b+q}_{\mcal{M}}(X)\to \cdots,
\end{align}

For a complex scheme $X$, the basic cofiber sequence of the classical equivariant homotopy theory given by $C _{2+}\rightarrow S^0\rightarrow S^\sigma$ induces a  l.e.s.  for the Bredon cohomology of the complex points of $X$. The complex realization $X\rightarrow X(\C)$ defined on $C _2Sm/\C\rightarrow C _2Top$ extends to a map $SH _{C _2}(\C)\rightarrow SH _{C _2}$ between the stable equivariant $\A^1-$ homotopy category over $\C$ and the classical stable equivariant homotopy category. According to \cite{HOV1} the complex realization of the motivic spectrum $\MA$ gives the spectrum that defines Bredon cohomology. The realization functor also induces a map between the two long exact sequences induced by the cofiber sequence $C _{2+}\rightarrow S^0\rightarrow S^\sigma$ in the motivic equivariant homotopy category and in the classical equivariant homotopy category.

The connecting map in the l.e.s. \ref{connec} is induced by the map $$S^\sigma\stackrel{\delta}{\rightarrow} C _{2+}\wedge S^\sigma=C _{2+}\wedge S^1.$$

 \begin{proposition} \label{comp}The composition $S^\sigma\stackrel{\delta}{\rightarrow} C _{2+}\wedge S^\sigma=C _{2+}\wedge S^1\stackrel{p}{\rightarrow} S^\sigma$ where the first map is the connecting map and the second map is induced by $C _{2+}\stackrel{p}{\rightarrow} S^0$ induces multiplication by 2 on Bredon motivic cohomology of a field i.e.
 $$H^{a+(p-1)\sigma,b+q\sigma} _{C _2}(k,\Z)\stackrel{p^*}{\rightarrow} H^{a+p-1,b+q}(k,\Z)\stackrel{\delta^*}{\rightarrow} H^{a+(p-1)\sigma,b+q\sigma} _{C _2}(k,\Z)$$
 is multiplication by 2 for any field $k$.
 \end{proposition}
 \begin{proof} The map $S^\sigma\stackrel{\delta}{\rightarrow} C _{2+}\wedge S^\sigma$ induces a stable map $S^0\rightarrow C _{2+}$ which is Spanier-Whitehead dual to the stable map $C _{2+}\stackrel{p}{\rightarrow} S^0$.
 \end{proof}
 \begin{corollary} The map on Bredon motivic cohomology induced by $C _{2+}\rightarrow S^0$ has a 2-torsion kernel.
 \end{corollary}
  \begin{proof} According to the above composition we have that if $x\in Ker(p^*)$ then $2x=0$.
  \end{proof}
We also recall that the topological Bredon cohomology groups of a point with $\Z/2$-coefficients are (\cite{dug:kr}):

$$
H^{a+p\sigma}_{Br}(pt, \ul{\Z/2}) =
\begin{cases}
\Z/2 & 0\leq -a \leq p\\
\Z/2 & 1 < a \leq -p \\ 
0 & \text{otherwise}
\end{cases}.
$$
With integer coefficients  (\cite{dug:kr}), we have that if  $p>0$ (the positive cone) and  $p$ is even then
$$
H^{a+p\sigma}_{Br}(pt, \ul{\Z}) =
\begin{cases}
\Z & a=-p\\
\Z/2 & -p < a \leq 0, a=\text{even} \\ 
0 & \text{otherwise}
\end{cases},
$$
and if $p>0$ (the positive cone) and  $p$ is odd then 
$$
H^{a+p\sigma}_{Br}(pt, \ul{\Z}) =
\begin{cases}
\Z/2 & -p < a \leq 0, a=\text{even} \\ 
0 & \text{otherwise}
\end{cases},
$$
and if  $p<0$ (the negative cone) and $p$ is even then
$$
H^{a+p\sigma}_{Br}(pt, \ul{\Z}) =
\begin{cases}
\Z/2 & 1< a < -p, a=\text{odd}\\
\Z & a=-p\\ 
0 & \text{otherwise}
\end{cases},
$$
and if $p<0$ (negative cone) and $p$ is odd then
$$
H^{a+p\sigma}_{Br}(pt, \ul{\Z}) =
\begin{cases}
\Z/2 & 1< a \leq -p, a=\text{odd}\\ 
0 & \text{otherwise}
\end{cases}.
$$

  \section{Complexes of equivariant equidimensional cycles}
 
  We can define a presheaf of $C _2-$equivariant equidimensional cycles on $C _2Sm/k$ for any smooth $C _2$-scheme as in the non-equivariant case: $$z _{equi}(T,r)^{C _2}(S)=\text{the free abelian group generated by elementary}$$ $$\text{equivariant $r-$equidimensional cycles over $S$ in $S\times T$.}$$ An elementary equivariant $r-$equidimensional cycle is of the form  $\tilde{Z}= Z+g _1Z+...+g _nZ$, where $Z$ is an irreducible closed subvariety of $S\times T$ which is dominant and equidimensional of relative dimension $r$ over a connected component of $S$ and $g _i$ are a set of coset representatives of   $Stab(Z)=\{g\in G|gZ=Z\}$. We conclude that for a projective or proper smooth $C _2-$scheme $T$ we have $z _{equi}(T,0)^{C _2}=\Z _{tr}(T)^{C _2}$. This is because $z _{equi}(T,0)(S)$ is the free abelian group generated by closed irreducible  subvarieties $Z$ of $S\times T$ which are quasi-finite and dominant over a connected component of $S$
  and coincides for a proper or projective $C _2-$scheme $T$ with $\Z _{tr}(T)(S)$.
 
 We define Suslin-Friedlander equivariant motivic complexes for a $C _2-$representation $V=a+b\sigma$ to be $$\Z^{SF}(V)=C _*z _{equi}(\A(V),0)^{C _2}[-2a-2b\sigma].$$ This is a complex of sheaves in the equivariant Nisnevich topology on $C _2Sm/k$.  
 
 The next theorem states that Suslin-Friedlander equivariant motivic complexes compute Bredon motivic cohomology. 
 \begin{theorem} \label{id} Let $k$ be a perfect field and $V$ a $C _2$-representation. Then there is a quasi-isomorphism of complexes of sheaves in the equivariant Zariski topology on $C _2Sm/k$, $$\Z(V)\simeq \Z^{SF}(V).$$ In particular
 $H^{a+p\sigma,b+q\sigma}(X,\Z)=H^{a}_{C_2Nis}(X, \Z^{SF}(V)[p\sigma])=H^{a}_{C_2Zar}(X, \Z^{SF}(V)[p\sigma])$ for any $C _2$ -representation $V=b+q\sigma$. 
 \end{theorem}
 \begin{proof} Let  $V=b+q\sigma$ a $C _2$-representation. As in the nonequivariant case (Theorem 16.8 \cite{MVW}) we let $$F _V(U)=\text{the free abelian group generated by the equivariant $0$-equidimensional (over $U$) closed}$$$$\text{ subvarieties  on $U\times \A(V)$ that doesn't touch $U\times 0$}.$$ We have that $F _V(U)\subset z _{equi}(A(V),0)^{C _2}(U)$. 
 
 We have a commutative diagram in the category of presheaves with equivariant transfers:
 \[
\xymatrixrowsep{0.3in}
\xymatrixcolsep {0.25in}
\xymatrix{
0\ar[r]&
\Z _{tr}(\P(V\oplus 1)\setminus\{0\})^{C _2}/\Z _{tr}(\P(V))^{C _2}\ar[r]\ar[d]& 
\Z _{tr}(\P(V\oplus 1))^{C _2}/\Z _{tr}(\P(V))^{C _2}\ar[r] \ar[d] &
coker _1 \ar[r]\ar[d]^{r}& 
0\\
0\ar[r] &
F _V \ar[r] & 
z _{equi}(\A(V),0)^{C _2}  \ar[r] &
coker _2 \ar[r] &
0.
}
\]
The vertical maps are injective by construction. The middle map is injective because we have an exact sequence $$0\rightarrow \Z _{tr}(\P(V))^{C _2}\rightarrow \Z _{tr}(\P(V\oplus 1))^{C _2}\rightarrow z _{equi}(\A(V),0)^{C _2}.$$ We see that $coker _1(U)$ is free abelian on equivariant correspondences $Z\subset U\times \P(V\oplus 1)$ which touch $U\times\{0\}$ and $coker _2(U)$ is free abelian on the equivariant closed $0$-equidimensional $W\subset U\times \A(V)$ which touch $U\times \{0\}$. But the middle map is injective on these generators so it implies that $r$ is injective. 

 We can prove that the map $r$ is surjective on semilocal affine Hensel schemes with a single closed orbit. Let $S$ be such a scheme and $Z\subset S\times \A(V)$ a $0-$equidimensional closed equivariant over $S$. This means that $Z$ is quasi-finite over $S$. This implies that $Z/C _2\rightarrow S/C _2$ is quasi-finite over the local affine Hensel scheme $S/C _2$. This implies that the quotient $Z/C _2=Z' _1\sqcup Z' _0$ is a disjoint union of a closed subscheme $Z' _1$ which is finite and surjective over $S/C _2$  and  a closed subscheme $Z' _0$ which doesn't contain any point over the closed point of the local Hensel scheme. Moreover $Z' _0$ doesn't intersect $S/{C _2}\times \{0\}$ (see Lemma 16.11 \cite{MVW}). Let $p:Z\rightarrow Z/C _2$ be the canonical quotient map which is finite and surjective and $Z _1=p^{-1}(Z' _1)$, $Z _0=p^{-1}(Z' _0)$. This implies that $Z=Z _0\sqcup Z _1$. We also have that $Z _0$ is a closed $C _2$-equivariant subscheme which doesn't intersect $S\times\{0\}$ and $Z _1$ is a $C _2-$equivariant closed subscheme which is finite and surjective over $S$.
 
    This implies that $Z _0$ is an $C _2-$equivariant cycle in $F _V(S)$. It is also clear that $Z _1$ comes from $\Z _{tr}(\P(V\oplus 1))^{C _2}/\Z _{tr}(\P(V))^{C _2}$ and $Z=Z _1$ in  $coker _2$. This implies $r$ is surjective on semilocal affine Hensel schemes with a single closed orbit. 
 
 We have that $C_*coker _1\simeq C _*coker _2$ is a quasi-isomorphism of sheaves in equivariant Zariski topology (see Proposition \ref{iso} below). 
 
 We notice that $\P(V\oplus1)\setminus\{0\}\hookrightarrow \P(V\oplus 1)$ is $C _2-$equivariantly $\A^1-$ homotopy equivalent to the inclusion $\P(V)\hookrightarrow \P(V\oplus 1)$ (see Page 5 \cite{HOV1})  so $$C _*(\Z _{tr}(\P(V\oplus 1)\setminus\{0\})^{C _2}/\Z _{tr}(\P(V))^{C _2})$$ is a chain contractible complex of presheaves. 
 
 To conclude the theorem, we notice that $C _*F _V$ is also a chain contractible complex of presheaves. Let $H$ be the following map $H _X: F _V(X)\rightarrow F _V(X\times \A^1)$ (with trivial action on $\A^1$) given by the equivariant pull-back of cycles $f: X\times \A(V)\times \A^1\rightarrow X\times \A(V)$, $f(x,r,t)=(x,rt)$ which is flat on $X\times \A(V)\setminus\{0\}$. Then $F _V(i _1)\circ H _X=id$ and $F _V(i _0)\circ H _X=0$ which implies that $C _*F _V$ is a chain contractible complex. Here $i _1$ is the inclusion $i _1:X\times \A(V)\rightarrow X\times \A(V)\times \A^1$ with $i _1(x,t)=(x,t,1)$ and similarly $i _0: X\times \A(V)\rightarrow X\times \A(V)\times \A^1$ with $i _0(x,t)=(x,t,0)$. 
 
 Notice that we only used equivariant $\mathbb{A}^1$- equivalence in order to prove that certain complexes of presheaves are chain contractible complexes.
 
Using that $C _*$ is an exact functor and using the five lemma we conclude that  $$C _*\Z _{tr}(\P(V\oplus 1))^{C _2}/\Z _{tr}(\P(V))^{C _2}\simeq C _*z _{equi}(\A(V),0)^{C _2}.$$
 \end{proof}
 \begin{proposition}\label{iso} Let $k$ be a perfect field. Let $F$ be a presheaf with equivariant transfers such that $F _{eNis}=0$. Then $C _*F _{eNis}\simeq 0$ and $C _*F _{eZar}\simeq 0$.
 \end{proposition}
 \begin{proof} We have that $C _*F _{eNis}\simeq 0$ iff $(H _i) _{eNis}=0$ where $H _i=H _i(C _*F _{eNis})=0$ for any i.  This is obviously true for $i\leq 0$, where for $i=0$ we used the condition $F _{eNis}=0$. We prove by induction that $(H _i) _{eNis}=0$ for any $i$. Suppose $(H _j) _{eNis}=0$ for any $j<i$. This means that the good truncation at level $i$, $\tau (C _*F _{eNis})$ is quasi-isomorphic to $C _*F _{eNis}$. We also have that the presheaf $H _i$ is homotopy invariant and then  $(H _i) _{eNis}$ is homotopy invariant (see Theorem 1.2. \cite{HOV}). This implies that  
 $$Hom _{D^- _{eNis}(k)}((C _*F) _{eNis}, (H _i)_{eNis}[i])=Hom _{D^- _{eNis}(k)}(F _{eNis}, (H _i) _{eNis}[i])=0,$$
 where $D^- _{eNis}(k)$ is the derived category of bounded above complexes of equivariant Nisnevich sheaves. This  implies that the map in $D^- _{eNis}(k)$, $$C _*F _{eNis}\simeq \tau C _*F _{eNis}\rightarrow (H _i) _{eNis}[i]$$ which induces an isomorphism on the $i^{th}$ homology groups is zero. This implies that $(H _i) _{eNis}=0$ and the induction is complete. 
 
  We have that $H _i(S)=0$ for any semilocal Hensel affine with a single closed orbit. In particular $H _i(\sqcup _j Spec(E _j))=0$ with a $\Z/2$-action on the disjoint union and $E _i$ finitely generated field extensions of $k$. For every semilocal affine with a single closed orbit $S$ over $k$ and  $S _0$ a $\Z/2$-invariant dense open subscheme we have that $F(S)\hookrightarrow F(S _0)$ for every homotopy invariant presheaf $F$ with equivariant transfers (Theorem 7.13, \cite{HOV}). Taking the intersection of all  affine $\Z/2-$equivariant open subsets of $S$ we obtain a disjoint union of  finitely generated field extensions of $k$. We conclude that $H _i(S)=0$ for every semi-local affine with a single closed orbit. Notice that any point on a quasi-projective $\Z/2$-variety has an affine invariant open neighborhood. This implies that $(H  _i) _{eZar}\simeq 0$ which implies $(C _*F) _{eZar}\simeq 0$.
 \end{proof} 
Part of the following proposition is in Proposition 3.15 \cite{HOV1}.
  \begin{proposition} \label{triv} Let $X$ be a smooth scheme with trivial $C _2$-action and $V=b+q\sigma$ a $C _2$-representation. Then
  $$\phi:H^{a+p\sigma,b+q\sigma}(X,\Z)\simeq H^{a-2b} _{Nis}(X,\Z(V)[p\sigma]).$$
    \end{proposition}
  \begin{proof} We also have the functor $(-)^e: (C _2Sm/k)\rightarrow (Sm/k)$ that forgets the $C _2$-action and sends an equivariant Nisnevich cover into a Nisnevich cover of the schemes without action giving a morphism of sites $(-)^e: (Sm/k) _{Nis}\rightarrow (C _2Sm/k) _{C _2Nis}$. This induces a map on cohomology $\phi: H^{a+p\sigma,b+q\sigma}(X,\Z)\rightarrow H^{a-2b} _{Nis}(X,\Z(V)[p\sigma])$. To construct the inverse of this map we look at the inclusion $t: (Sm/k)\rightarrow (C _2Sm/k)$ which gives a morphism of sites $t: (C _2Sm/k) _{C _2Nis}\rightarrow (Sm/k) _{Nis}$ with $t^*$ an exact functor and $t^*(\Z(V)[p\sigma])=\Z(V)[p\sigma]$ (see Proposition 3.15\cite{HOV1} and Lemma 3.19\cite{HOV}). From the Leray spectral sequence we have that $$\psi: H^{a-2b} _{Nis}(X,\Z(V)[p\sigma])\simeq H^{a+p\sigma,b+q\sigma}(X,\Z)=H^{a-2b} _{eNis}(X, \Z(V)[p\sigma])$$
  is an isomorphism; this is the inverse of $\phi$.
  \end{proof}

 According to Theorem 5.4 \cite{Nie} we have that $$C _*z _{equi}(\A(nk[C _2]),0)^{C _2}=\oplus _{j=n}^{2n-1}\Z/2(j)[2j]\oplus \Z(2n)[4n]$$
 is a decomposition in $DM^-(k)$, where $k$ is a field of characteristic zero. Here $k[C _2]=1\oplus \sigma$ is the regular representation of $C_2$. 
 
 We conclude that $C _*z _{equi}(\A(n\sigma),0)^{C _2}(n)[2n]=\oplus _{j=n}^{2n-1}\Z/2(j)[2j]\oplus \Z(2n)[4n]$ so $$C _*z _{equi}(\A(n\sigma),0)^{C _2}=\oplus _{j=0}^{n-1}\Z/2(j)[2j]\oplus \Z(n)[2n].$$
 
  \begin{proposition} \label{Nie}  Let $k$ be a field of characteristic zero  and $b,q\geq 0$. Then  $$H^{a+2q\sigma,b+q\sigma}_{C_2}(k,\Z)= \oplus _{j=0}^{q-1} H^{a+2j,j+b}(k, \Z/2) \oplus H^{a+2q, b+q}(k, \Z).$$
 \end{proposition}
 \begin{proof}
 Using Theorem \ref{id}  we have that
   $$H^{a+2q\sigma,b+q\sigma}_{C_2}(k,\Z)= H^{a-2b}(k, C _*z _{equi}(\A(V),0)^{C_2})=$$
   $$=H^{a-2b} (k, C _*z _{equi}(\A(q\sigma),0)^{C _2}(b)[2b])= H^{a-2b} (k, \oplus _{j=0}^{q-1}\Z/2(j+b)[2j+2b]\oplus \Z(q+b)[2q+2b])=$$
   $$= \oplus _{j=0}^{q-1} H^{a-2b}(k, \Z/2(j+b)[2j+2b]) \oplus H^{a-2b} (k, \Z(q+b)[2q+2b])= $$
    $$= \oplus _{j=0}^{q-1} H^{a+2j,j+b}(k, \Z/2) \oplus H^{a+2q, b+q} (k, \Z).$$
    
    We used that $$C _*z _{equi}(\A(b+q\sigma),0)^{C _2}\simeq C _*z _{equi}(\A(q\sigma),0)^{C _2}\otimes^{tr} C _*z _{equi}(\A^b,0)=C _*z _{equi}(\A(q\sigma),0)^{C _2}(b)[2b]$$ for any $b,q\geq 0$. 
   \end{proof}
 
 With $\Z/2-$coefficients (the same for $\Z/l$) we have (Theorem 5.7 \cite {Nie}):
 $$C _*z(\mathbb{A}(n\sigma))^{\Z/2}\otimes\Z/2= \oplus _{j=0}^{n-1}(\Z/2(j)[2j]\oplus\Z/2(j)[2j+1])\oplus \Z/2(n)[2n].$$
This implies the following proposition:
\begin{proposition} \label {fin} For any field $k$ of characteristic zero and any  $b,q\geq 0$ we have $$H^{a+2q\sigma,b+q\sigma} (k,\Z/2)=\oplus _{j=0}^{q-1}(H^{a+2j,j+b} (k,\Z/2)\oplus H^{a+2j+1,j+b}(k,\Z/2))\oplus$$ 
$$\oplus H^{a+2q,q+b} (k,\Z/2)\simeq \oplus _{j=0}^{2q}H^{a+j} _{et}(k,\mu _2^{\otimes{b+j}}).$$
The statement is also valid with $\Z/l$-coefficients.
\end{proposition}
\begin{proof} The first statement follows from V.Voevodsky \cite{Voev:miln} and the statement with $\Z/l$ coefficients follows from M.Rost and V.Voevodsky's Bloch-Kato theorem \cite{Voev:BK}, \cite{nr}.
\end{proof}
In the case of $n=1$ we have  $C _*z _{equi}(\mathbb{A}(\sigma))^{\Z/2}\otimes\Z/2=\Z/2\oplus\Z/2[1]\oplus \Z/2(1)[2]$ and $[C _*z _{equi}(\mathbb{A}(\sigma))^{\Z/2}\otimes\Z/2](k)=\Z/2\oplus\Z/2[1]\oplus k^*/k^{*2}[1]\oplus \Z/2[2].$ We also have that $C _*z _{equi}(\mathbb{A}(\sigma))^{\Z/2}\simeq \Z/2\oplus O^*[1]$. These remarks are used in Proposition \ref{2q}. 

Remark that Proposition \ref{fin} and Proposition \ref{Nie} are valid for any smooth scheme with a trivial action.

 \section{Bredon motivic cohomology of a field in weight 0}
 In this section we will compute Bredon motivic cohomology in weight less than or equal to zero of an arbitrary  field $k$.

We have the following vanishing theorem for Bredon motivic cohomology groups:
\begin{proposition} \label{van} For any smooth $C _2-$variety $X$ and any $b+q<0, b<0$  we have that  $$H^{a+p\sigma,b+q\sigma} _{C _2}(X,\Z)=0.$$ 
  \end{proposition}
  \begin{proof} We use the motivic isotropy cofiber sequence $\EG C _{2+}\rightarrow S^0\rightarrow \EGt C _2$ and the periodicity of Bredon motivic cohomology of $\EGt C _{2}$ and $\EG C _2$ to obtain $$H^{a,b}(\EGt C _2\times X)=0\rightarrow H^{a+p\sigma,b+q\sigma} _{C _2}(X)\rightarrow H^{a+2q+(p-2q)\sigma, b+q} _{C _2}(\EG C _2\times X)=0.$$ This  implies the vanishing of the statement. The last equality follows from the fact that $H^{a+p\sigma, b}(\EG C _2\times X)=0$ for any $a, p$ and $b<0$. This is the result of using inductively the l.e.s.
  $$H^{a,b}(\EG C _2\times X)=0\rightarrow H^{a,b} _{C _2}(\EG C _2\times X)=0\rightarrow H^{a+\sigma,b} _{C _2}(\EG C _2\times X)\rightarrow H^{a+1,b} (\EG C _2\times X)=0$$
and the l.e.s.
 $$H^{a-1,b}(\EG C _2\times X)=0\rightarrow H^{a-\sigma,b} _{C _2}(\EG C _2\times X)\rightarrow H^{a,b} _{C _2}(\EG C _2\times X)=0\rightarrow H^{a,b}(\EG C _2\times X)=0.$$

 \end{proof}

 \begin{proposition} Suppose $b+q<0$. Then $$H^{a+p\sigma,b+q\sigma} _{C _2}(k,A)=H^{a+2q\sigma,b+q\sigma} _{C _2}(k,A)$$ for any abelian group $A$. When $A=\Z/2$ and  $a>b\geq 0$ or $a\leq -2$ we have that 
 $$H^{a+p\sigma,b+q\sigma} _{C _2}(k,\Z/2)=H^{a+2q\sigma,b+q\sigma} _{C _2}(k,\Z/2)=(H^{*} _{et}(k,\mu^* _2)[\tau,u,v]/(u^2=v\tau))^{(a-1,b)}.$$
 \end{proposition}
\begin{proof}  We restrict to the case $q<0, b\geq 0$ because if $q<0$ and $b<0$ we can use Proposition \ref{van} to conclude that both groups are zero. Using the cofiber sequence $C _{2+}\rightarrow S^0\rightarrow S^\sigma$ and $b+q< 0$ we conclude that $H^{a+p\sigma,b+q\sigma} _{C _2}(k,A)=H^{a+2q\sigma,b+q\sigma} _{C _2}(k,A)$ for any $p$ because motivic cohomology is zero in negative weight.

From the definition, it follows that  $H^{a+2q\sigma,b+q\sigma} _{C _2}(k,A)=\rH^{a} _{C _2}(T^{-q\sigma},A(b))$. We have a cofiber sequence $\mathbb{A}(-q\sigma)\setminus\{0\}_+\rightarrow S^0\rightarrow T^{-q\sigma}$ which implies that 
$$H^{a,b}(k,A)\rightarrow H^{a}_{C _2}(\mathbb{A}(-q\sigma)\setminus\{0\},A(b))\rightarrow \rH^{a+1} _{C _2}(T^{-q\sigma},A(b))\rightarrow H^{a+1,b}(k,A)$$ with $H^{a}_{C _2}(\mathbb{A}(-q\sigma)\setminus\{0\},A(b))=H^{a,b}(\mathbb{A}(-q\sigma)\setminus\{0\}/C _2,A)$ because $\mathbb{A}(-q\sigma)\setminus\{0\}$ has a free $C _2$-action. 

We know that $H^{*,*}(\mathbb{A}(-q\sigma)\setminus\{0\}/C _2,\Z/2)$ is a $H^{*,*}(k,\Z/2)=H^{*} _{et}(k,\mu^* _2)[\tau]$-module generated by $uv^i$ and $v^i$ where $u$ is a $(1,1)$ element and $v$ is a $(2,1)$ element and $\tau$ is a $(0,1)$ element (Theorem 6.10 \cite{Voc}). We also have the relation $u^2=v\tau$. We can write the equality as $$H^{a+p\sigma,b+q\sigma} _{C _2}(k,\Z/2)=H^{a+2q\sigma,b+q\sigma} _{C _2}(k,\Z/2)=(H^{*} _{et}(k,\mu^* _2)[\tau,u,v]/(u^2=v\tau))^{(a-1,b)},$$ for any $b+q<0$ and $a>b$ or $a\leq -2$.
\end{proof}
\begin{proposition}\label{case0} We have  $\rH^{a+p\sigma,0} _{C _2}(\EGt C_2,A)=0$ for any $a,p$ and any abelian group $A$.
\end{proposition}
\begin{proof} From the periodicities $(\sigma-1,\sigma-1)$ and $(\sigma-1,0)$ of $\EGt C _2$ (Proposition 2.9 \cite{HOV1}) we have that $\rH^{a+p\sigma,0} _{C _2}(\EGt C_2)=\rH^{a,0} _{C _2}(\EGt C_2)$. We  have from the motivic isotropy sequence that 
$$0\rightarrow\rH^{a,0} _{C _2}(\EGt C _2)\rightarrow H^{a,0} _{C _2}(k)\simeq H^{a,0} _{C _2}(\EG C _2)\rightarrow \rH^{a+1,0} _{C _2}(\EGt C_2)\rightarrow 0$$ because $H^{a,0} _{C _2}(k)=H^{a,0}(k)\simeq H^{a,0} _{C _2}(\EG C _2)=H^{a,0} (\BG C _2)$ for any $a$ (which also implies that the left map of the sequence is injective). In the case of $a\neq 0$ both sides of the isomorphism are zero. When $a=0$ the isomorphism is given by the map $H^{0,0}(k,\Z)=\Z\rightarrow H^{0,0}(\BG C _2,\Z)=\Z$ that sends a generator to a generator. This implies that $\rH^{a,0} _{C _2}(\EGt C _2)=0$ for any $a$. 
\end{proof}

\begin{proposition} \label{0b} Let  $b+q=0$, $b\leq 0$. Then $H^{a+p\sigma,b+q\sigma} _{C _2}(k,\Z)=H^{a+2q+(p-2q)\sigma,0} _{C _2}(k,\Z).$
\end{proposition}
\begin{proof} We use the motivic isotropy cofiber sequence $\EG C_{2+}\rightarrow S^0\rightarrow \EGt C_2$. Using Proposition \ref{case0} we get for $b\leq 0$ that $$\rH^{a,b} _{C _2}(\EGt C_2)=0\rightarrow H^{a+p\sigma,b+q\sigma} _{C _2}(k)\rightarrow H^{a+p\sigma,b+q\sigma} _{C _2}(\EG C _2)\rightarrow \rH^{a+1,b}  _{C _2}(\EGt C _2)=0$$ so we have $H^{a+p\sigma,b+q\sigma}  _{C _2}(k)=H^{a+p\sigma,b+q\sigma}  _{C _2}(\EG C _2)=H^{a+2q+(p-2q)\sigma,0}  _{C _2}(\EG C _2)$ from the $(2\sigma-2,\sigma-1)$ periodicity of $\EG C_2$ and $b+q=0$. This means that $$H^{a+p\sigma,b+q\sigma}  _{C _2}(k)=H^{a+2q+(p-2q)\sigma,0}  _{C _2}(\EG C _2)=H^{a+2q+(p-2q)\sigma,0}  _{C _2}(k).$$ 
\end{proof}
We will compute below  the abelian groups $H^{a+p\sigma,0}  _{C _2}(k,\Z)$. We have the following l.e.s.
$$\rightarrow H^{a+p,0} (k,\Z)\rightarrow H^{a+p\sigma,0}  _{C _2}(k,\Z)\rightarrow H^{a+(p+1)\sigma,0}  _{C _2}(k,\Z)\rightarrow H^{a+p+1,0} (k,\Z)\rightarrow.$$
It implies that $H^{a+p\sigma,0} _{C _2}(k,\Z)\simeq H^{a-(a+1)\sigma,0} _{C _2}(k,\Z)$ for any $p\leq -a-1$ and that $H^{a+p\sigma,0} _{C _2}(k,\Z)\simeq H^{a-(a-1)\sigma,0}  _{C _2}(k,\Z)$ for any $p\geq -a+1$. It is enough to compute the groups $H^{a-(a+1)\sigma,0} _{C _2}(k,\Z)$, $H^{a-a\sigma,0} _{C _2}(k,\Z)$ and $H^{a-(a-1)\sigma,0} _{C _2}(k,\Z)$ for any $a$. We have a l.e.s. between these groups for any $a$
$$0\rightarrow H^{a-(a+1)\sigma,0} _{C _2}(k,\Z)\rightarrow H^{a-a\sigma,0} _{C _2}(k,\Z)\rightarrow H^{0,0} _{C _2}(k,\Z)\rightarrow $$
$$\rightarrow H^{a+1-(a+1)\sigma,0} _{C _2}(k,\Z)\rightarrow H^{a+1-a\sigma,0} _{C _2}(k,\Z)\rightarrow 0.$$
Notice that by definition $H^{a-a\sigma,0} _{C _2}(k,\Z)=H^a _{C _2}(k,Z _{top}(-a\sigma))$, $H^{a+1-a\sigma,0} _{C _2}(k,\Z)=H^{a+1} _{C _2}(k,Z _{top}(-a\sigma))$ and $H^{a-(a+1)\sigma,0} _{C _2}(k,\Z)=H^{a} _{C _2}(k,Z _{top}((-a-1)\sigma))$.
\begin{proposition} \label{neg} We have that $H^{a-(a-1)\sigma,0} _{C _2}(k,\Z)=0$ for any $a\geq 1$ or $a=-1$ and $\Z/2$ if $a=0$. Also $H^{a-(a+1)\sigma,0} _{C _2}(k,\Z)=0$ for $a\leq 1$.
\end{proposition}
\begin{proof} The complex $Z _{top}(n\sigma)=Z _{top}(\sigma)\otimes ^{tr}...\otimes^{tr}Z _{top}(\sigma)$ with $n>0$ is a complex of presheaves  on the positions $-n,...,-1,0$ and it satisfies descent in Nisnevich topology on $Sm/k$. Here $Z _{top}(\sigma)$ is the complex $\Z _{tr}(C _2)^{C _2}\rightarrow \Z _{tr}(k)^{C _2}$ on the positions $-1$ and $0$. If $n=0$ then $Z _{top}(0)=\Z$ on the position zero. If $n<0$ then $Z _{top}(n\sigma)=Z _{top}(-\sigma)\otimes ^{tr}...\otimes^{tr}Z _{top}(-\sigma)$ on the positions $0,1...-n$ with $Z _{top}(-\sigma)$ being the complex $\Z _{tr}(k)^{C _2}\rightarrow \Z _{tr}(C _2)^{C _2}$ on the positions $0,1$. This implies that $H^{a-(a-1)\sigma,0} _{C _2}(k)=H^{-b+1}_{C_2}(k,Z _{top}(b\sigma))=0$ if $b\leq 0$ (we wrote $b=-a+1$). It implies that $H^{a-(a-1)\sigma,0} _{C _2}(k)=0$ for $a\geq 1$ . If $a=0$ then 
$H^{\sigma,0} _{C _2}(k,\Z)=H^0(Z _{top}(\sigma)(k))$ where $Z _{top}(\sigma)(k)$ is the complex $\Z\stackrel{2}{\rightarrow}\Z$ on the positions $-1,0$. It implies that $H^{\sigma,0} _{C _2}(k,\Z)=\Z/2$. If $a=-1$ then we have the exact sequence $$H^{0,0} _{C _2}(k,\Z)\rightarrow H^{-1+\sigma,0} _{C _2}(k,\Z)=0\rightarrow H^{-1+2\sigma,0} _{C _2}(k,\Z)\rightarrow H^{1,0} _{C _2}(k,\Z)=0$$ so we conclude that 
$H^{-1+2\sigma,0} _{C _2}(k,\Z)=0$.

We also have that $H^{a-(a+1)\sigma,0} _{C _2}(k)=H^{-b-1}_{C_2}(k,Z _{top}(b\sigma))=0$ if $b\geq 0$ (we wrote $b=-a-1$). It implies that $H^{a-(a+1)\sigma,0} _{C _2}(k)=0$ for $a\leq -1$. If $a=0$ then $$H^{-\sigma,0} _{C _2}(k,\Z)=H^0(Z _{top}(-\sigma)(k))=0$$ because $Z _{top}(-\sigma)(k)$ is the complex $\Z\stackrel{id}{\rightarrow} \Z$ on the positions $0$ and $1$. If $a=1$ then we have a sequence $$0\rightarrow H^{1-2\sigma,0} _{C _2}(k,\Z)\rightarrow H^{1-\sigma,0} _{C _2}(k,\Z)=0\rightarrow H^{0,0} _{C _2}(k,\Z)=H^{0,0} (k,\Z)=\Z$$ so we conclude that $H^{1-2\sigma,0} _{C _2}(k,\Z)=0$.
\end{proof}
In the next proposition we will study the cohomology of the complexes $Z _{top}(\pm \sigma)$ and $Z _{top}(\pm 2\sigma)$.
\begin{proposition}\label{case2}
 We have that $H^{a+2\sigma,0}_{C _2}(k, \Z)=0$ if $a\neq -2,0$ and  $H^{2\sigma,0}_{C _2}(k, \Z)=\Z/2$, $H^{-2+2\sigma,0}_{C _2}(k, \Z)=\Z$. We also have $H^{2-2\sigma,0}_{C _2}(k, \Z)=\Z$ and $H^{a-2\sigma,0}_{C _2}(k, \Z)=0$ for any $a\neq 2$.

We have $H^{\sigma,0}_{C _2}(k, \Z)=\Z/2$, $H^{a+\sigma} _{C _2}(k,\Z)=0$ for $a\neq 0$ and $H^{a-\sigma,0}_{C _2}(k, \Z)=0$ for any $a$.
\end{proposition}
\begin{proof}
We have that $H^{a-a\sigma,0}_{C _2}(k)=H^a(Z _{top}(-a\sigma)(k))$. If $a=-1$ then the complex is $\Z\stackrel{2}{\rightarrow} \Z$ on the positions $-1$ and $0$ so  $H^{\sigma,0}_{C _2}(k,\Z)=H^{0}(Z _{top}(\sigma)(k))=\Z/2$ and $H^{a+\sigma,0}_{C _2}(k,\Z)=0$ for $a\neq 0$. If $a=1$ then the complex is $\Z\stackrel{id}{\rightarrow} \Z$ on the positions $0,1$ so $H^{a-\sigma,0}_{C _2}(k,\Z)=H^{a}(Z _{top}(-\sigma)(k))=0$ for any $a$. 

Let's compute the differentials of the complex $Z _{top}(2\sigma)$. We have the projections $q _i: C _2\times C _2\rightarrow C _2$  for $i=1,2$ and $p:C _2\rightarrow pt$. They induce the maps $q _i: \Z _{tr}(C _2\times C _2)^{C _2}\rightarrow \Z _{tr}(C _2)^{C _2}$ and $q^t _i: \Z _{tr}(C _2)^{C _2}\rightarrow \Z _{tr}(C _2\times C _2)^{C _2}$ and $p: \Z _{tr}(C _2)^{C _2}\rightarrow \Z _{tr}(k)^{C _2}$ and $p^t: \Z _{tr}(k)^{C _2}\rightarrow \Z _{tr}(C _2)^{C _2}$. The action of $C _2$ on $C _2=\{0,1\}$ switches the copies of $k$ i.e. $0\rightarrow 1$ and $1\rightarrow 0$ and induces a pointwise action on $C _2\times C _2$ i.e $(0,0)\rightarrow (1,1)$, $(0,1)\rightarrow (1,0)$, $(1,0)\rightarrow (0,1)$, $(1,1)\rightarrow (0,0)$. The map $p: \Z _{tr}(C _2)^{C _2}(k)\rightarrow \Z _{tr}(k)(k)$ is $p(a)=2a$ because the map is induced by $p: \Z _{tr}(C _2)(k)\rightarrow \Z _{tr}(k)(k)$ which is $p(a,b)=(a,b)(1,1)^T=a+b$. 

The dual  $p^t: \Z _{tr}(k)(k)\rightarrow \Z _{tr}(C _2)(k)$ is given by $p(a)=(a,a)$ which is given by the transpose matrix $(1,1)$.  If we look at $q^t _1: \Z _{tr}(C _2)(k)\rightarrow \Z _{tr}(C _2\times C _2)(k)$ we see that is given by $q^t _1(a,b)=(a,b,a,b)=(a,b)(I _2,I _2)$ which on fixed points gives $q^t _1: \Z _{tr}(C _2)(k)^{C_2}\rightarrow \Z _{tr}(C _2\times C _2)(k)^{C_2}$ with $q^t _1(a)=(a,a)$. Looking at the transpose we obtain that $q _i: \Z _{tr}(C _2\times C _2)\rightarrow \Z _{tr}(C _2)$, $q _i(a,b,c,d)=(a+c,b+d)=(a,b,c,d)(I _2,I _2)^T$ implying that on fixed points we get $q _i: \Z _{tr}(C _2\times C _2)^G\rightarrow \Z _{tr}(C _2)^G$, $q _i(a,b)=a+b$. 

 In conclusion we have the following complex 
$$0\rightarrow \Z _{tr}(C _2\times C _2)^{C _2}\rightarrow \Z _{tr}(C _2)^{C _2}\oplus  \Z _{tr}(C _2)^{C _2}\rightarrow \Z _{tr}(k)\rightarrow 0$$
with the first map of the complex given by $(q _1, -q _2)$ i.e. $g(a,b)=(a+b,-a-b)$ and the second map is $f(a,b)=2a+2b=p+p$.  It implies that $H^{-2+2\sigma,0}(k,\Z)=\Z$, $H^{-1+2\sigma,0}(k,\Z)=0$ and $H^{2\sigma,0}(k,\Z)=H^{\sigma,0}(k,\Z)=\Z/2$.

The complex $Z _{top}(-2\sigma)$ is the following complex
$$0\rightarrow \Z _{tr}(k)\rightarrow \Z _{tr}(C _2)^{C _2}\oplus \Z _{tr}(C _2)^{C_2}\rightarrow \Z _{tr}(C _2\times C _2)^G\rightarrow 0$$
with the first map being $g(a)=(a,a)$ and the second map $f(a,b)=(a-b,a-b)$. It implies that $H^{2-2\sigma,0}(k,\Z)=\Z$ and $H^{n-2\sigma,0}(k,\Z)=0$ for $n\neq 2$.

Notice that we can compute the cohomology of the complexes $Z _{top}(\pm 2\sigma)$ directly from the cohomology of the complexes $Z _{top}(\pm \sigma)$ (without the need to compute the differentials in $Z _{top}(\pm 2\sigma)$). We have the following l.e.s.
$$0\rightarrow H^{1-2\sigma,0} _{C _2}(k,\Z)\rightarrow H^{1-\sigma,0} _{C _2}(k,\Z)=0\rightarrow H^{0,0} _{C _2}(k,\Z)\rightarrow$$
$$\rightarrow H^{2-2\sigma,0} _{C _2}(k,\Z)\rightarrow H^{2-\sigma,0} _{C _2}(k,\Z)=0\rightarrow 0.$$
It implies that $H^{1-2\sigma,0} _{C _2}(k,\Z)=0$ and $H^{2-2\sigma,0} _{C _2}(k,\Z)=\Z$. In conclusion $H^{a-2\sigma,0} _{C _2}(k,\Z)=0$ if $a\neq 2$.
 We have the following l.e.s.
 $$0\rightarrow H^{-2+\sigma,0} _{C _2}(k,\Z)=0\rightarrow H^{-2+2\sigma,0} _{C _2}(k,\Z)\rightarrow H^{0,0} _{C _2}(k,\Z)=\Z\rightarrow$$
$$\rightarrow H^{-1+\sigma,0} _{C _2}(k,\Z)=0\rightarrow H^{-1+2\sigma,0} _{C _2}(k,\Z)\rightarrow 0.$$

This implies that $H^{-2+2\sigma,0} _{C _2}(k,\Z)=\Z$ and $H^{-1+2\sigma,0} _{C _2}(k,\Z)=0$. We also have $H^{2\sigma,0} _{C _2}(k,\Z)=H^{\sigma,0} _{C _2}(k,\Z)=\Z/2$ and
$H^{a+2\sigma,0} _{C _2}(k,\Z)=0$ if $a\neq -2,0$.
\end{proof}

We compute below the negative cone of weight 0 Bredon motivic cohomology.
\begin{proposition} If $n=even\geq 4$ then 
$$H^{m-n\sigma,0}_{C _2}(k, \Z) =
\begin{cases}
\Z/2 & 3\leq m<n, m=odd \\
\Z& m=n\\ 
0& otherwise\\
\end{cases}.$$

If $n=odd\geq 3$ then $$H^{m-n\sigma,0}_{C _2}(k, \Z) =
\begin{cases}
\Z/2 & 3\leq m\leq n, m=odd \\ 
0& otherwise\\
\end{cases}.$$
\end{proposition}
\begin{proof}
We compute the negative cone of weight 0 Bredon motivic cohomology by induction. 

 If $n=3$ we have a l.e.s
$$0\rightarrow H^{2-3\sigma,0}_{C_2}(k,\Z)\rightarrow H^{2-2\sigma,0}_{C_2}(k,\Z)\stackrel{f _1}{\rightarrow} H^{0,0}(k,\Z)\rightarrow H^{3-3\sigma,0}_{C_2}(k,\Z)\rightarrow H^{3-2\sigma,0}(k,\Z)=0$$
and because $$H^{1-\sigma,0} _{C _2}(k,\Z)=0\rightarrow H^{0,0}(k,\Z)=\Z\simeq H^{2-2\sigma,0}_{C_2}(k,\Z)=\Z\rightarrow H^{2-\sigma,0} _{C _2}(k,\Z)=0$$  we have, according to Proposition \ref{comp}, that $f _1$ is multiplication by 2. It implies $H^{3-3\sigma,0}_{C_2}(k,\Z)=\Z/2$ and $H^{m-3\sigma,0}_{C_2}(k,\Z)=0$ if $m\neq 3$. We also have the sequence 
$$0\rightarrow H^{3-4\sigma,0}_{C_2}(k,\Z)\rightarrow H^{3-3\sigma,0}_{C_2}(k,\Z)=\Z/2\stackrel{0}{\rightarrow} H^{0,0}(k,\Z)=\Z\rightarrow$$
$$\rightarrow H^{4-4\sigma,0}_{C_2}(k,\Z)\rightarrow H^{4-3\sigma,0}_{C_2}(k,\Z)=0$$
which implies that $H^{3-3\sigma,0}_{C_2}(k,\Z)=H^{3-4\sigma,0}_{C_2}(k,\Z)=H^{3-n\sigma,0}_{C_2}(k,\Z)=\Z/2$ for any $n\geq 3$ and $H^{m-3\sigma,0}_{C_2}(k,\Z)=H^{m-4\sigma,0}_{C_2}(k,\Z)=0$ for $m\leq 2$ and $m\geq 5$ and $H^{4-4\sigma}_{C_2}(k,\Z)=\Z$. In conclusion 
$$H^{n-4\sigma,0}_{C _2}(k, \Z) =
\begin{cases}
0 & n\leq 2, n\geq 5\\
\Z/2 & n=3 \\
\Z& n=4\\ 
\end{cases}.$$
We have 
$$0\rightarrow H^{4-5\sigma,0} _{C _2}(k,\Z)\rightarrow H^{4-4\sigma,0} _{C _2}(k,\Z)=\Z\stackrel{f _2}{\rightarrow} H^{0,0}(k,\Z)=\Z\rightarrow$$ 
$$\rightarrow H^{5-5\sigma,0} _{C _2}(k,\Z)\rightarrow H^{5-4\sigma,0} _{C _2}(k,\Z)=0\rightarrow$$ and that  
$H^{0,0}(k,\Z)=\Z\simeq H^{4-4\sigma,0} _{C _2}(k,\Z)=\Z$ is an isomorphism from the above l.e.s.. We conclude that $f _2$ is multiplication by 2 (according to Proposition \ref{comp}) so $H^{4-5\sigma,0} _{C _2}(k,\Z)=0$ and $H^{5-5\sigma} _{C _2}(k,\Z)=\Z/2$. It also implies that $H^{m-5\sigma,0} _{C _2}(k,\Z)=0$ for any $m\neq 3,5$ and $H^{3-5\sigma,0} _{C _2}(k,\Z)=H^{3-4\sigma,0} _{C _2}(k,\Z)=\Z/2$. 

We will compute the case $n=6$ and conclude by induction the general case. 

For $n=6$ we have that 
$$0\rightarrow H^{5-6\sigma,0} _{C _2}(k,\Z)\rightarrow H^{5-5\sigma,0} _{C _2}(k,\Z)=\Z/2\stackrel{0}{\rightarrow} H^{0,0}(k,\Z)=\Z\rightarrow$$
$$\rightarrow H^{6-6\sigma} _{C _2}(k,\Z)\rightarrow H^{6-5\sigma} _{C _2}(k,\Z)=0.$$  It implies that $H^{5-6\sigma,0}(k,\Z)\simeq H^{5-5\sigma,0}(k,\Z)=\Z/2$ and $H^{0,0}(k,\Z)=\Z\simeq H^{6-6\sigma}(k,\Z)$. 

We conclude that $H^{6-6\sigma}(k,\Z)=\Z$, $H^{5-6\sigma}(k,\Z)=\Z/2$, $H^{4-6\sigma}(k,\Z)=0$, $H^{3-6\sigma}(k,\Z)=\Z/2$ and $H^{m-6\sigma}(k,\Z)=0$ for $m\leq 2$ and $m\geq 7$. 

By induction we conclude that if $n=even\geq 4$ we have $H^{n-n\sigma}(k,\Z)=\Z$, $H^{n-1-n\sigma}(k,\Z)=\Z/2$,  $H^{n-2-n\sigma}(k,\Z)=0$, $H^{n-3-n\sigma}(k,\Z)=\Z/2$... $H^{2-n\sigma}(k,\Z)=0$ and $H^{m-n\sigma}(k,\Z)=0$ for any $m\leq 2$ or $m\geq n+1$. 

If $n=odd\geq 3$ then $H^{n-n\sigma}(k,\Z)=\Z/2$, $H^{n-1-n\sigma}(k,\Z)=0$, $H^{n-2-n\sigma}(k,\Z)=\Z/2$, ... , $H^{3-n\sigma}(k,\Z)=\Z/2$ and $H^{m-n\sigma}(k,\Z)=0$ if $m\leq 2$ or $m\geq n+1$.
\end{proof}
The proof of the above proposition computes completely the negative cone of Bredon motivic cohomology in weight 0 with integer coefficients. 

The positive cone of Bredon motivic cohomology in weight 0 with integer coefficients is completely computed in the next corollary:
\begin{corollary} \label{c33} If $n=even\geq 2$ then 
$$H^{m+n\sigma,0}_{C _2}(k, \Z) =
\begin{cases}
\Z/2 & -n< m\leq 0, m=even \\
\Z& m=-n\\ 
0& otherwise\\
\end{cases}.$$

If $n=odd\geq 1$ then $$H^{m+n\sigma,0}_{C _2}(k, \Z) =
\begin{cases}
\Z/2 & -n< m\leq 0, m=even \\ 
0& otherwise\\
\end{cases}.$$
\end{corollary} 
\begin{proof} If $n=3$ we have 
$$0\rightarrow H^{-3+2\sigma,0} _{C _2}(k,\Z)=0\rightarrow H^{-3+3\sigma,0} _{C _2}(k,\Z)\rightarrow $$
$$\rightarrow H^{0,0}(k,\Z)=\Z\stackrel{g _1}{ \rightarrow} H^{-2+2\sigma,0} _{C _2}(k,\Z)=\Z \rightarrow H^{-2+3\sigma,0}_{C _2}(k,\Z)\rightarrow 0.$$
We have that $$0=H^{-2+\sigma}(k,\Z)\rightarrow H^{-2+2\sigma,0} _{C _2}(k,\Z)=\Z\simeq H^{0,0}(k,\Z)=\Z\rightarrow H^{-1+\sigma}(k,\Z)=0$$ so we conclude that $g_1$ is multiplication by 2. It implies that $H^{-2+3\sigma,0}_{C _2}(k,\Z)=\Z/2$ and $H^{-3+3\sigma,0} _{C _2}(k,\Z)=0$. We also have  $H^{3\sigma,0}_{C _2}(k,\Z)=H^{2\sigma,0}_{C _2}(k,\Z)=\Z/2$ (Proposition \ref{case2}) and  $H^{m+3\sigma,0}_{C _2}(k,\Z)=0$ if $m\neq -2,0$. When $n=4$ we have 
$$0\rightarrow H^{-4+3\sigma,0} _{C _2}(k,\Z)=0\rightarrow H^{-4+4\sigma,0} _{C _2}(k,\Z)\rightarrow $$
$$\rightarrow H^{0,0}(k,\Z)=\Z \rightarrow H^{-3+3\sigma,0} _{C _2}(k,\Z)=0\rightarrow H^{-3+4\sigma,0}_{C _2}(k,\Z)\rightarrow 0,$$
which implies that $H^{-4+4\sigma,0}_{C _2}(k,\Z)=\Z$, $H^{-3+4\sigma,0}_{C _2}(k,\Z)=0$,  $H^{-2+4\sigma,0}_{C _2}(k,\Z)=H^{-2+3\sigma,0}_{C _2}(k,\Z)=\Z/2$, $H^{-1+4\sigma,0}_{C _2}(k,\Z)=0$ and $H^{4\sigma,0}_{C _2}(k,\Z)=H^{3\sigma,0}_{C _2}(k,\Z)=\Z/2$ and the rest of the groups are zero. In conclusion
$$H^{m+4\sigma,0}_{C _2}(k, \Z) =
\begin{cases}
\Z& m=-4\\
\Z/2 & -2\leq m\leq 0, m=even \\ 
0& otherwise\\
\end{cases}.$$
In the case of $n=5$ we have the following 
$$0\rightarrow H^{-5+4\sigma,0} _{C _2}(k,\Z)=0\rightarrow H^{-5+5\sigma,0} _{C _2}(k,\Z)\rightarrow $$
$$\rightarrow H^{0,0}(k,\Z)=\Z \stackrel{g _2}{\rightarrow} H^{-4+4\sigma,0} _{C _2}(k,\Z)=\Z\rightarrow H^{-4+5\sigma,0}_{C _2}(k,\Z)\rightarrow 0.$$
The map $g _2$ is multiplication by $2$ because
$$0=H^{-4+3\sigma,0} _{C _2}(k,\Z)\rightarrow H^{-4+4\sigma,0} _{C _2}(k,\Z)\simeq H^{0,0}(k,\Z)\rightarrow H^{-3+3\sigma} _{C _2}(k,\Z)=0.$$ In conclusion $H^{-5+5\sigma,0} _{C _2}(k,\Z)=0$, $H^{-4+5\sigma,0}_{C _2}(k,\Z)=\Z/2$ and $H^{-3+5\sigma,0}_{C _2}(k,\Z)=H^{-3+4\sigma,0} _{C _2}(k,\Z)=0$, $H^{-2+5\sigma,0}_{C _2}(k,\Z)=H^{-2+4\sigma,0} _{C _2}(k,\Z)=\Z/2$, $H^{-1+5\sigma,0}_{C _2}(k,\Z)=0$ and $H^{5\sigma,0}_{C _2}(k,\Z)=\Z/2$ and the rest are zero. Now by induction we conclude the rest of the cases.
\end{proof}
From the discussion above we also conclude:
\begin{corollary} \label{zero}
We have 
$$
H^{a-a\sigma,0}_{C _2}(k, \Z) =
\begin{cases}
\Z & a= \text{even}\\
0 & a=\text{odd}, a\leq 1 \\
\Z/2& a=\text{odd}, a>1\\ 
\end{cases}.
$$
We also have $H^{a-a\sigma,0}_{C _2}(k,\Z/2)=\Z/2$ for any $a\leq 0$, $a\geq 2$ and zero for $a=1$.
\end{corollary}

The above theorems conclude that Bredon motivic cohomology of a field with integer coefficients in weight $0$ coincide (as abstract groups) with Bredon cohomology of a point with integer coefficients. We also have the following theorem :
\begin{theorem} \label{compZ} Let $k$ be an arbitrary field. Then with $\Z/2-$coefficients we have that 
$$H^{a+p\sigma,0} _{C _2}(k,\Z/2)=
\begin{cases}
\Z/2& 1<a\leq -p\\
\Z/2 & 0\leq -a\leq p \\ 
0& otherwise\\
\end{cases}.$$
\end{theorem}

 \section{Bredon motivic cohomology of a field in weights 1 and $\sigma$}
 In this section we assume that $k$ is a field of characteristic zero.
 
We have the following Bredon motivic cohomology of $\EGt C _2$ in weight 1.
\begin{proposition} \label{case1} We have $\rH^{a+p\sigma,1}_{C_2}(\EGt C _2,\Z)=0$ if $a\neq 3$ and $\rH^{3+p\sigma,1}_{C_2}(\EGt C _2,\Z)=\Z/2$.
\end{proposition}
\begin{proof} We have that the motivic isotropy sequence gives a l.e.s
$$0\rightarrow \rH^{1,1} _{C _2}(\EGt C _2)\rightarrow H^{1,1}_{C_2}(k,\Z)=k^*\stackrel{\alpha}{\rightarrow} H^{1,1}_{C_2}(\EG C _2,\Z)=k^*\rightarrow \rH^{2,1}_{C_2}(\EGt C _2,\Z)\rightarrow$$
$$\rightarrow H^{2,1} _{C_2}(k,\Z)=0\rightarrow H^{2,1}_ {C_2}(\EG C _2,\Z)\rightarrow \rH^{3,1}_{C_2}(\EGt C _2,\Z)\rightarrow H^{3,1}_{C_2}(k,\Z)=0.$$
The map $\alpha:k^*=O^*(k)\rightarrow k^*=O^*(\BG C _2)$ is the identity and  $H^{2,1}_{C_2}(\EG C _2,\Z)=H^{2,1}(\BG C _2,\Z)=\Z/2$ (Corollary 3.5 \cite{T}). It implies that $\rH^{m,1} _{C_2}(\EGt C _2,\Z)=0$ if $m\neq 3$ and $\rH^{3,1}_{C_2}(\EGt C _2,\Z)=\Z/2$. The general result follows from the periodicities of $\EGt C _2$. 
\end{proof}
\begin{proposition}\label{2q} a) $H^{2\sigma,\sigma} _{C _2}(k,\Z)=\Z/2$, $H^{2\sigma-1,\sigma}_{C _2}(k,\Z)= k^*$,  $H^{2\sigma+n,\sigma}_{C _2}(k,\Z)=0$ for any $n\neq 0,-1$. 
 
 b) $H^{2\sigma,\sigma}_{C _2}(k,\Z/2)=\Z/2$,  $H^{2\sigma-1,\sigma}_{C _2}(k,\Z/2)=\Z/2\oplus k^*/k^{*2}$, $H^{2\sigma-2,\sigma}_{C _2}(k,\Z/2)=\Z/2$ and $H^{2\sigma+n,\sigma}_{C _2}(k,\Z/2)=0$ for any $n\neq 0,-1,-2$.
 \end{proposition} 
 \begin{proof} From Proposition \ref{Nie} we have $$H^{2\sigma,\sigma}_{C _2}(k,\Z)=H^0(z _{equi}(A(\sigma),0)^{C_2}(k))=H^0(\Z/2[0]\oplus k^*[1])=\Z/2.$$ With $\Z/2$-coefficients we have
 $$H^{2\sigma,\sigma}_{C _2}(k,\Z/2)=H^0(z _{equi}(A(\sigma),0)^{C _2}(k)\otimes\Z/2)=H^0((\Z/2[0]\oplus k^*[1])\otimes \Z/2)=$$$$=H^0(\Z/2[0]\oplus \Z/2[1]\oplus k^*[1]\otimes \Z/2)=
 H^0(\Z/2[0]\oplus \Z/2[1]\oplus \Z/2[2]\oplus k^* /k^{*2}[1])=\Z/2.$$ 
 \end{proof}
 
  \begin{proposition} \label{sigma} a) We have $H^{\sigma,\sigma}_{C _2}(k,\Z)=\Z/2$ and $H^{\sigma+n,\sigma}_{C _2}(k,\Z)=0$ for any $n\neq 0$. With $\Z/2-$coefficients we have 
$H^{\sigma,\sigma}_{C _2}(k,\Z/2)=\Z/2$ and $H^{\sigma-1,\sigma}_{C _2}(k,\Z/2)=\Z/2$ and $H^{\sigma+n,\sigma}_{C _2}(k,\Z/2)=0$ for any $n\neq 0,-1$.

b) $H^{1,\sigma}_{C _2}(k,\Z)=k^{*2}$ and $H^{n,\sigma}_{C _2}(k,\Z)=0$ for any $n\neq 1$; 

With $\Z/2-$coefficients we have
\begin{equation}
H^{0,\sigma}_{C _2}(k,\Z/2)= \left\{
\begin{array}{ll}
      \Z/2 &  \sqrt{-1}\neq\emptyset\\
      0 & otherwise\\
\end{array} 
\right. 
\end{equation}
and $H^{1,\sigma}_{C _2}(k,\Z/2)=k^{*2}/k^{*4}$. 

In the case $\sqrt{-1}\neq \emptyset$ we have that $k^*/k^{*2}\stackrel{\times 2}{\simeq} k^{*2}/k^{*4}=H^{1,\sigma}_{C _2}(k,\Z/2)$. Also $H^{n,\sigma}_{C _2}(k,\Z/2)=0$ for $n\neq 0,1$.
\end{proposition}

\begin{proof} 
a) We have the following l.e.s. 
$$H^{0,1} (k,\Z)=0\rightarrow H^{\sigma-1,\sigma} _{C _2}(k,\Z)\rightarrow $$
$$H^{2\sigma-1,\sigma} _{C _2}(k,\Z)=k^*\rightarrow H^{1,1}(k,\Z)=k^*\rightarrow H^{\sigma,\sigma} _{C _2}(k,\Z)\rightarrow H^{2\sigma,\sigma} _{C _2}(k,\Z)=\Z/2 \rightarrow H^{2,1}(k,\Z)=0$$
We have that  $$id: H^{2\sigma-1,\sigma} _{C _2}(k,\Z)=H^0 _{C _2}(k,O^*)=O^*(k)=k^*\rightarrow H^{1,1}(k,\Z)=k^*.$$
This is because the connecting map $H^{2\sigma-1,\sigma} _{C _2}(k,\Z)=H^{-1} _{C _2}(k, \Z(\sigma)[2\sigma])=H^{-1} _{C _2}(k,O^*[1]\oplus \Z/2[0])=H^0 _{C _2}(k,O^*)=k^*\rightarrow H^{1,1} (k,\Z)=H^{2\sigma-1,\sigma} _{C _2}(C_2,\Z)=H^{-1} _{C _2}(C _2, \Z(\sigma)[2\sigma])=H^0 _{C _2}(C _2,O^{*C_2})=k^*$ is the identity map being induced by the diagonal $k^*\rightarrow (k^*\times k^*)^{C _2}=k^*$. It implies from the above l.e.s. that $H^{\sigma-1,\sigma} _{C _2}(k,\Z)=0$ and $H^{\sigma,\sigma} _{C _2}(k,\Z)\simeq H^{2\sigma,\sigma} _{C _2}(k,\Z)=\Z/2.$ 
We also have that  $$0=H^{-1,1} (k)\rightarrow H^{\sigma-2,\sigma}_{C _2}(k)\rightarrow H^{2\sigma-2,\sigma}_{C _2}(k,\Z)=0$$ which implies $H^{\sigma-2,\sigma}_{C _2}(k,\Z)=0$. Similarly $H^{\sigma +n,\sigma}_{C _2}(k,\Z)\simeq H^{2\sigma+n,\sigma}_{C _2}(k,\Z)=0$ for any $n\geq 1$ and $n\leq -2$.

 From the split universal coefficients sequence we have that $$H^{\sigma,\sigma}_{C _2}(k,\Z/2)=H^{\sigma,\sigma}_{C _2}(k,\Z)\otimes \Z/2=\Z/2$$ and $$H^{\sigma-1,\sigma}_{C _2}(k,\Z/2)= {_2H}^{\sigma,\sigma}_{C _2}(k,\Z)=\Z/2.$$ 
 
b) We have 
$$H^{0,1} (k,\Z)=0\rightarrow H^{0,\sigma}_{C _2}(k,\Z)\rightarrow H^{\sigma,\sigma}_{C _2}(k,\Z)=\Z/2\stackrel{\alpha}{\rightarrow} H^{1,1} (k,\Z)=k^*\stackrel{\delta}{\rightarrow} $$
$$H^{1,\sigma} _{C _2}(k,\Z)\rightarrow H^{\sigma+1,\sigma} _{C _2}(k,\Z)=0\rightarrow H^{2,1} (k,\Z)=0\rightarrow H^{2,\sigma} _{C _2}(k,\Z)\rightarrow H^{2+\sigma,\sigma} _{C _2}(k,\Z)=0\rightarrow 0.$$ 
It implies that $H^{n,\sigma} _{C _2}(k,\Z)=0$ for any $n\neq 0,1$. The map $\alpha$ is either injective or zero. We prove that the map $\alpha$ is an injective map.

We have $H^1 _{C _2}(S^\sigma,O^{*C_2})=\Z/2$ and $H^i_{C _2}(S^\sigma,O^{*C_2})=0$ for any $i\neq 1$. This is because we have
$$H^{n+\sigma,\sigma} _{C _2}(k,\Z)=H^{n}_{C _2}(k, (O^{*C_2}[1]\oplus \Z/2)[-\sigma])=$$
$$=H^{n}_{C _2}(S^{\sigma}, O^{*C_2}[1]\oplus \Z/2)=H^{n+1}_{C _2}(S^{\sigma},O^{*C_2})\oplus H^{n}_{C _2}(S^{\sigma}, \Z/2).$$
This is zero if $n\neq 0$ and $\Z/2$ if $n=0$ from point a). We have that $H^n_{C _2}(S^{\sigma}, \Z/2)=0$ for any $n$ because the complexes $Z _{top}(-\sigma)$ have trivial cohomology (see proof of Proposition \ref{neg}). Obviously $H^1_{C _2}(S^1,O^{*C_2})=k^*$ and $H^n_{C _2}(S^1,O^{*C_2})=0$ for $n\neq 1$.
The map $\alpha$ is the inclusion map  being given by the forgetting action map $$\Z/2=H^1_{C _2}(S^\sigma,O^{*C_2})\rightarrow \rH^1_{C _2}(C _{2+}\wedge S^\sigma, O^{*C_2})=H^1_{C _2}(S^1,O^{*C_2})=k^*.$$ The injectivity follows from the following diagram:   

\[
\xymatrixrowsep{0.1in}
\xymatrixcolsep {0.1in}
\xymatrix{
H^{0}_{C _2}(k,O^{*C_2})\ar[r]\ar[d]^r&
H^{0}_{C _2}(C _2,O^{*C_2}) \ar[r]\ar[d]^p & 
H^{1}_{C _2}(S^\sigma,O^{*C_2})\ar[r] \ar[d]^\alpha &
H^{1}_{C _2}(k,O^{*C_2}) \ar[r]\ar[d]^q & 
H^{1}_{C _2}(C _2,O^{*C_2})\ar[d]\\
H^{0}_{C _2}(C _2,O^{*C_2})\ar[r] &
\rH^{0}_{C _2}(C _{2+}\wedge C _{2+},O^{*C_2})  \ar[r] & 
\rH^{1}_{C _2}(C _{2+}\wedge S^\sigma, O^{*C_2})  \ar[r] &
H^{1}_{C _2}(C _2,O^{*C_2}) \ar[r] &
\rH^{1}_{C _2}(C _{2+}\wedge C _{2+},O^{*C_2}).
}
\]

The map $p$ is injective and the maps $r$ and $q$ are  bijectives  (Lemma 3.19, \cite{HOV}) which implies that the map $\alpha$ is injective. 

Because we have a short exact sequence of abelian groups
$$0\rightarrow \Z/2\rightarrow k^*\stackrel{\times 2}{\rightarrow} k^{*2}\rightarrow 0$$
it implies that $H^{1,\sigma}_{C _2}(k,\Z)=k^{*2}$ and $H^{0,\sigma}_{C _2}(k,\Z)=0$.

We have  $H^{0,\sigma}_{C _2}(k,\Z/2)={_2H}^{1,\sigma}_{C _2}(k,\Z)=$  all elements $x^2$ of order 2 in $k^*$. 

It implies that $H^{0,\sigma}_{C _2}(k,\Z/2)=\Z/2$ (if $-1\in k^{*2}$) or $H^{0,\sigma}_{C _2}(k,\Z/2)=0$ (if $-1\notin k^{*2}$).  Also $$H^{-1,\sigma}_{C _2}(k,\Z/2)={_2H}^{0,\sigma} _{C _2}(k,\Z)=0.$$

 With $\Z/2-$coefficients we have  the l.e.s. 
 $$\rightarrow H^{-1,\sigma}_{C _2}(k,\Z/2)=0\rightarrow H^{\sigma-1,\sigma}_{C _2}(k,\Z/2)=\Z/2 \simeq H^{0,1}(k,\Z/2)=\Z/2\stackrel{0}{\rightarrow} H^{0,\sigma}_{C _2}(k,\Z/2)\rightarrow$$
 $$\rightarrow H^{\sigma,\sigma}_{C _2}(k,\Z/2)=\Z/2\stackrel{\alpha}{\rightarrow} H^{1,1}(k,\Z/2)=k^*/k^{*2}\stackrel{\delta}{\rightarrow} H^{1,\sigma}_{C _2}(k,\Z/2)\rightarrow 0.$$
 If $-1\in k^{*2}$ then $H^{0,\sigma}_{C _2}(k,\Z/2)=\Z/2$ and the map $\alpha$ is zero which implies $H^{1,\sigma}_{C _2}(k,\Z/2)=k^*/k^{*2}$. If $-1\notin k^{*2}$ then $H^{0,\sigma}_{C _2}(k,\Z/2)=0$ and the map $\alpha$ is injective and from the short exact sequence
$$0\rightarrow \Z/2\rightarrow k^*/k^{*2}\rightarrow H^{1,\sigma}_{C _2}(k,\Z/2)\rightarrow 0,$$ we have that 
$H^{1,\sigma}_{C _2}(k,\Z/2)=k^{*2}/k^{*4}$ with the connecting homomorphism $\delta$ being multiplication by $2$. 
\end{proof}

\begin{proposition} \label{prop}
 a) $H^{\sigma,1} _{C _2}(k,\Z)=\Z/2$, $H^{\sigma+1,1}_{C _2}(k,\Z)=k^*/k^{*2}$, $H^{\sigma+n,1}_{C _2}(k,\Z)=0$ for $n\neq 0,1$.
                               
                               b) $H^{\sigma,1}_{C _2}(k,\Z/2)=\Z/2\oplus k^*/k^{*2}$, $H^{\sigma+1,1}_{C _2}(k,\Z/2)=k^*/k^{*2}$, $H^{\sigma-1,1}_{C _2}(k,\Z/2)=\Z/2$, $H^{\sigma+n,1}_{C _2}(k,\Z/2)=0$ for $n\neq 0,-1,1$.

\end{proposition}
\begin{proof} We have a l.e.s. 

$$H^{0,1}(k,\Z)=0\rightarrow H^{0,1} _{C _2}(k,\Z)=0\rightarrow H^{\sigma,1} _{C _2}(k,\Z)\rightarrow H^{1,1}(k,\Z)=k^*\stackrel{\delta}{\rightarrow}$$
$$\stackrel{\delta}{\rightarrow} H^{1,1} _{C _2}(k,\Z)=k^*\rightarrow H^{1+\sigma,1}_{C _2}(k,\Z)\rightarrow $$

$$\rightarrow H^{2,1}(k,\Z)=0\rightarrow H^{2,1} _{C _2}(k,\Z)=0\rightarrow H^{2+\sigma,1}_{C _2}(k,\Z)\rightarrow$$
$$\rightarrow H^{3,1}(k,\Z)=0\rightarrow H^{3,1} _{C _2}(k,\Z)=0\rightarrow H^{3+\sigma,1} _{C _2}(k,\Z)\rightarrow$$
We obtain $H^{\sigma+n,1}_{C _2}(k,\Z)=0$ for $n\neq 1,0$.

We prove now that the connecting map $\delta$ is the multiplication by $2$. According to Proposition \ref{comp} the composition 
$$H^{1,1} _{C _2}(k,\Z)\rightarrow H^{1,1} (k,\Z)\stackrel{\delta}{\simeq} H^{1,1} _{C _2}(k,\Z)$$
is multiplication by 2. We have the following l.e.s.
$$0\rightarrow H^{1-\sigma,1} _{C _2}(k,\Z)\rightarrow H^{1,1} _{C _2}(k,\Z)\rightarrow H^{1,1}(k,\Z)\rightarrow H^{2-\sigma,1} _{C _2}(k,\Z)\rightarrow 0.$$
Using the motivic isotropy sequence and the periodicity of Bredon motivic cohomology of $\EG \Z/2$ (see also Proposition \ref{borel}) we obtain from Proposition \ref{sigma} that
$$H^{1-\sigma,1} _{C _2}(k,\Z)=H^{1-\sigma,1} _{C _2}(\EG\Z/2,\Z)=H^{\sigma-1,\sigma} _{C _2}(\EG\Z/2,\Z)=H^{\sigma-1,\sigma} _{C _2}(k,\Z)=0.$$
We also have that
$$0\rightarrow H^{2-\sigma,1} _{C _2}(k,\Z)\rightarrow H^{2-\sigma,1} _{C _2}(\EG\Z/2)\rightarrow \Z/2\rightarrow H^{3-\sigma,1} _{C _2}(k,\Z)=0$$
with $H^{2-\sigma,1} _{C _2}(\EG\Z/2)=H^{\sigma,\sigma} _{C _2}(\EG\Z/2)=H^{\sigma,\sigma} _{C _2}(k,\Z)=\Z/2$. It implies that $H^{2-\sigma,1} _{C _2}(k,\Z)=0$ and thus the map $H^{1,1} _{C _2}(k,\Z)\rightarrow H^{1,1} (k,\Z)$ is an isomorphism. This implies that the connecting map $\delta$ is multiplication by 2.

We obtain $H^{\sigma+1,1}_{C _2}(k,\Z)=k^*/k^{*2}$ and $H^{\sigma,1}_{C _2}(k,\Z)=\Z/2$.

With $\Z/2-$coefficients we have that $$H^{\sigma,1}_{C _2}(k,\Z/2)=H^{\sigma,1}_{C _2}(k,\Z)\otimes\Z/2\oplus {_2H}^{1+\sigma,1}_{C _2}(k,\Z)=\Z/2\oplus k^*/k^{*2}.$$ We also have 
$H^{1+\sigma,1}_{C _2}(k,\Z/2)=H^{1+\sigma,1}_{C _2}(k,\Z)\otimes\Z/2\oplus {_2H}^{2+\sigma,1}_{C _2}(k,\Z)=k^*/k^{*2}$ and $H^{-1+\sigma,1}_{C _2}(k,\Z/2)=H^{-1+\sigma,1}_{C _2}(k,\Z)\otimes\Z/2\oplus {_2H}^{\sigma,1}_{C _2}(k,\Z)=\Z/2.$ The rest of groups are obviously zero.
\end{proof}
In conclusion we have:
\begin{corollary}\label{gh} For any field $k$ of $char(k)=0$ we have 

$H^{\sigma,\sigma}_{C _2}(k,\Z)=\Z/2$, $H^{1,\sigma}_{C _2}(k,\Z)=k^{*2}$, $H^{\sigma,1}_{C _2}(k,\Z)=\Z/2$, $H^{1,1}_{C_2}(k,\Z)=k^*$. 

With $\Z/2-$coefficients we have $H^{\sigma,\sigma}_{C _2}(k,\Z/2)=\Z/2$, $H^{1,\sigma}_{C _2}(k,\Z/2)=k^{*2}/k^{*4}$, $H^{\sigma,1}_{C _2}(k,\Z/2)=\Z/2\oplus k^*/k^{*2}$, $H^{1,1}(k,\Z/2)=k^*/k^{*2}$.
\end{corollary}
In the following theorem we compute the positive cone of Bredon motivic cohomology of a quadratically closed field or a euclidian field (of characteristic zero).
\begin{theorem}  \label{com} Let $k$ be a quadratically closed field or a euclidian field. If $n\geq 2$ is even, then
 $$
H^{n\sigma-m,1}_{C _2}(k, \Z) =
\begin{cases}
k^* & m=n-1\\
k^*/k^{*2} & -1\leq m<n-1, m=\text{odd} \\
\Z/2& 0\leq m< n-1, m=\text{even}\\
0& otherwise\\ 
\end{cases}.
$$
If $n\geq 1$ is odd, then
 $$
H^{n\sigma-m,1}_{C _2}(k, \Z) =
\begin{cases}
k^*/k^{*2} & -1\leq m<n-1, m=\text{odd} \\
\Z/2& 0\leq m \leq n-1, m=\text{even}\\
0& otherwise\\ 
\end{cases}.
$$
\end{theorem}
\begin{proof}We have the following l.e.s. $$0 \rightarrow H^{2\sigma-1,1} _{C _2}(k,\Z)\rightarrow H^{1,1}(k,\Z)=k^*\stackrel{\alpha}{\rightarrow} H^{\sigma,1} _{C _2}(k,\Z)=\Z/2\rightarrow$$ $$ H^{2\sigma,1} _{C _2}(k,\Z)\rightarrow H^{2,1}(k,\Z)=0.$$ 
 Here $H^{\sigma,1} _{C _2}(k,\Z)=\Z/2$ from Proposition \ref{prop}. The composition 
$H^{\sigma,1} _{C _2}(k,\Z)\rightarrow H^{\sigma,1} _{C _2}(C _2,\Z)\rightarrow H^{\sigma,1} _{C _2}(k,\Z)$ is multiplication by 2 (the vertical maps in the diagram below are given by forgetting the action):
$$
\xymatrix{
\rH^{2\sigma,1} _{C _2}(S^\sigma)\ar@{^{(}->}[r]\ar@{^{(}->}[d]& \rH^{2\sigma, 1}_{C_2}(C_{2\,+}\wedge S^1) \ar[r]^-{\alpha}\ar[d] & \rH^{2\sigma, 1}_{C_{2}}(S^\sigma) \ar[d]\ar@{^{(}->}[d] \\
\rH^{2,1}_{\mcal{M}}(S^1) \ar[r] & \rH^{2,1}_{\mcal{M}}(S^1\vee S^1) \ar[r] &
\rH^{2,1}_{\mcal{M}}(S^1).
} 
$$
The first upper horizontal map is induced by the fold map $C _{2+}\wedge S^\sigma\rightarrow S^{\sigma}$ and the second is induced by the pinch map $S^\sigma\rightarrow C _+\wedge S^{1}$.
 
It implies that the homomorphism  $\alpha: k^*\rightarrow \Z/2$ has the property that $\alpha(1)=\alpha(-1)=1$. If $k$ is quadratically closed (i.e. $k=k^2$) then $\alpha$ is the zero map. If $k$ is a euclidian field (i.e. formally real and $k^*/k^{*2}\simeq \Z/2$) it implies that the map $\alpha$ is zero. Indeed suppose that $\alpha$ is surjective. Then $H=Ker(\alpha)$ is a subgroup of index 2 of $k^*$ that contains $k^{*2}$. It implies that $H=k^{*2}$, but $-1\in H$ because $\alpha(-1)=0$. This is contradiction because $k$ is formally real.

 Then we have $H^{2\sigma,1} _{C _2}(k,\Z)\simeq H^{\sigma,1} _{C _2}(k,\Z)=\Z/2$ and $H^{2\sigma-1,1} _{C _2}(k,\Z)=k^*$. By induction we have $H^{n\sigma,1} _{C _2}(k,\Z)\simeq H^{\sigma,1} _{C _2}(k,\Z)=\Z/2$ for any $n\geq 1$. We also have $H^{n\sigma+1,1} _{C _2}(k,\Z)=H^{2\sigma+1,1} _{C _2}(k,\Z)=H^{\sigma+1,1} _{C _2}(k,\Z)=k^*/k^{*2}$. 
 
 According to the sequence 
$$H^{n-1+s,1} _{C _2}(k,\Z)\rightarrow H^{(n-1)\sigma+s,1} _{C _2}(k,\Z)\rightarrow H^{n\sigma+s,1} _{C _2}(k,\Z)\rightarrow H^{n+s,1}(k,\Z)$$
and by induction we obtain that $H^{n\sigma+s,1} _{C _2}(k,\Z)=0$ if $s\geq 2$ and $n\geq 1$. This follows from  $H^{n\sigma+s,1} _{C _2}(k,\Z)=0$ if $n=0$ and $s\geq 2$. 
We also have 
$$0\rightarrow H^{2\sigma-2,1} _{C _2}(k,\Z)=0\rightarrow H^{3\sigma-2,1} _{C _2}(k,\Z)\rightarrow H^{1,1}(k,\Z)\stackrel{\psi}{\rightarrow} H^{2\sigma-1,1} _{C _2}(k,\Z)\rightarrow H^{3\sigma-1,1} _{C _2}(k,\Z)\rightarrow 0.$$
The map $\psi$ is multiplication by 2 because we have from the above that $$H^{\sigma-1,1} _{C _2}(k,\Z)=0\rightarrow H^{2\sigma-1,1} _{C _2}(k,\Z)\simeq H^{1,1}(k,\Z)\rightarrow 0$$ and the map on the $(3,1)-$ motivic cohomology induced by $S^1\wedge S^1\rightarrow S^1\wedge S^1\vee S^1\wedge S^1\rightarrow S^1\wedge S^1$ is given by multiplication by 2. This is shown in the diagram below where the vertical maps are forgetting action maps and the lower horizontal composition is the multiplication by 2.
$$
\xymatrix{
\rH^{3\sigma,1}_{C _2}(S^\sigma\wedge S^1)\ar[d] _{\iso}\ar[r]^{\iso} & \rH^{3\sigma, 1}_{C_2}(S^\sigma\wedge S^1\wedge C_{2\,+}) \ar[r]^-{\psi}\ar[d] & \rH^{3\sigma, 1}_{C_{2}}(S^{\sigma}\wedge S^1) \ar[d] _{\iso} \\
\rH^{3,1}_{\mcal{M}}(S^1\wedge S^1) \ar[r] & \rH^{3,1}_{\mcal{M}}(S^1\wedge S^1\vee S^1\wedge S^1) \ar[r] &
\rH^{3,1}_{\mcal{M}}(S^1\wedge S^1).
} 
$$
In conclusion we have $H^{3\sigma-2,1} _{C _2}(k,\Z)=\Z/2$ and $H^{3\sigma-1,1} _{C _2}(k,\Z)=k^*/k^{*2}=H^{n\sigma-1,1} _{C _2}(k,\Z)$ for any $n\geq 3$. 
 According to the above sequences we have $$H^{2\sigma-s,1} _{C _2}(k,\Z)=H^{\sigma-s,1} _{C _2}(k,\Z)=0$$ for any $s\geq 2$. We have $H^{3\sigma-s,1} _{C _2}(k,\Z)=H^{2\sigma-s,1} _{C _2}(k,\Z)=0$ for any $s\geq 3$. We have 
$$0\rightarrow H^{3\sigma-3,1} _{C _2}(k,\Z)=0\rightarrow H^{4\sigma-3,1} _{C _2}(k,\Z)\rightarrow H^{1,1} _{C _2}(k,\Z)\rightarrow H^{3\sigma-2,1} _{C _2}(k,\Z)\rightarrow H^{4\sigma-2,1} _{C _2}(k,\Z)\rightarrow 0$$
The middle map  $H^{1,1} _{C _2}(k,\Z)=k^*\rightarrow H^{3\sigma-2,1} _{C _2}(k,\Z)=\Z/2$ is zero from the argument above. We use that the map $H^{2\sigma-2,1} _{C _2}(k,\Z)=0\rightarrow H^{3\sigma-2,1} _{C _2}(k,\Z)\rightarrow H^{1,1}(k,\Z)$ is injective and that we have the following commutative diagram with the lower horizontal composition being multiplication by 2.
$$
\xymatrix{
\rH^{4\sigma-2,1} _{C _2}(S^\sigma)\ar@{^{(}->}[r]\ar@{^{(}->}[d]& \rH^{4\sigma-2, 1}_{C_2}(C_{2\,+}\wedge S^1) \ar[r]\ar[d] & \rH^{4\sigma-2, 1}_{C_{2}}(S^\sigma) \ar[d]\ar@{^{(}->}[d] \\
\rH^{2,1}_{\mcal{M}}(S^1) \ar[r] & \rH^{2,1}_{\mcal{M}}(S^1\vee S^1) \ar[r] &
\rH^{2,1}_{\mcal{M}}(S^1).
} 
$$

Then we have $H^{4\sigma-3,1} _{C _2}(k,\Z)=k^*$ and $H^{4\sigma-2,1} _{C _2}(k,\Z)=\Z/2$.  We have $H^{4\sigma-s,1} _{C _2}(k,\Z)=H^{3\sigma-s,1} _{C _2}(k,\Z)=0$ for any $s\geq 4$.

Let $n$ be even. By induction we have  $H^{n\sigma-m,1} _{C _2}(k,\Z)=k^*/k^{*2}$ if $-1\leq m<n-1$, $m$ odd and $H^{n\sigma-n+1,1} _{C _2}(k,\Z)=k^*$ and $H^{n\sigma-m,1} _{C _2}(k,\Z)=\Z/2$ for $0\leq m<n-1$, $m$ even.

Let $n$ be odd. Then $H^{n\sigma-m,1} _{C _2}(k,\Z)=k^*/k^{*2}$ if $-1\leq m\leq n-2$, $m$ odd and $H^{n\sigma-m,1} _{C _2}(k,\Z)=\Z/2$ if $-1\leq m\leq n$, $m$ even. 
\end{proof}

The negative cone in weight $1$ of Bredon motivic cohomology of a formally real field or a quadratically closed field of characteristic zero is computed below.
\begin{theorem} \label{neg1} Let $k$ be a formally real field or a quadratically closed field.
 
If $n\geq 2$ is even then
 $$
H^{m-n\sigma,1}_{C _2}(k, \Z) =
\begin{cases}
k^* & m=n+1\\
k^*/k^{*2} & 2<m\leq n, m=\text{even} \\
\Z/2& 2<m\leq n, m=\text{odd}\\
0& otherwise\\ 
\end{cases}.
$$
If $n\geq 1$ is odd then 
$$
H^{m-n\sigma,1}_{C _2}(k, \Z) =
\begin{cases}
k^*/k^{*2} & 2<m\leq n+1, m=\text{even} \\
\Z/2& 2<m\leq n+1, m=\text{odd}\\
0& otherwise\\ 
\end{cases}.
$$
If $n=0$ then $H^{m,1} _{C _2}(k,\Z)=H^{m,1}(k,\Z)=k^*$ if $m=1$ and $H^{m,1} _{C _2}(k,\Z)=0$ if $m\neq 1$.
\end{theorem}
\begin{proof} 

We have that $H^{m-n\sigma,1} _{C _2}(k,\Z)=H^{m,1}_{C_2}(k,\Z)=0$ for any $m\leq 0$ and $n\geq 0$. We have 
$$0\rightarrow H^{1-\sigma,1} _{C _2}(k,\Z)=0\rightarrow H^{1,1} _{C _2}(k,\Z)=k^*\simeq H^{1,1}(k,\Z)=k^*\rightarrow H^{2-\sigma,1} _{C _2}(k,\Z)=0\rightarrow 0,$$
from the proof of Proposition \ref{prop}. It implies that $H^{m-\sigma,1} _{C _2}(k,\Z)=0$ for any $m$. We have 
$$H^{2-\sigma,1} _{C _2}(k,\Z)=0\rightarrow H^{1,1}(k,\Z)\rightarrow H^{3-2\sigma,1} _{C _2}(k,\Z)\rightarrow 0$$
that implies $H^{3-2\sigma,1} _{C _2}(k,\Z)=k^*$ and $H^{m-2\sigma,1} _{C _2}(k,\Z)=0$ for $m \neq 3$. It also implies that $H^{2-n\sigma,1} _{C _2}(k,\Z)=0$ for any $n\geq 1$.

 We have $$0\rightarrow H^{3-3\sigma,1} _{C _2}(k,\Z)\rightarrow H^{3-2\sigma,1} _{C _2}(k,\Z)=k^*\rightarrow H^{1,1}(k,\Z)=k^*\rightarrow H^{4-3\sigma,1} _{C _2}(k,\Z)\rightarrow 0$$
and the middle map is multiplication by 2 because the composition $$H^{3-2\sigma,1} _{C _2}(k,\Z)=k^*\rightarrow H^{1,1}(k,\Z)=k^*\simeq H^{3-2\sigma,1} _{C _2}(k,\Z)=k^*$$ is multiplication by 2 from Proposition \ref{comp}.

In conclusion we obtain $H^{3-3\sigma,1} _{C _2}(k,\Z)=\Z/2$ and $H^{4-3\sigma,1} _{C _2}(k,\Z)=k^*/k^{*2}$ and $H^{n-3\sigma,1} _{C _2}(k,\Z)=0$ for $n\neq 3,4$. 

If $n=4$ we have that $H^{m-4\sigma,1} _{C _2}(k,\Z)=H^{m-3\sigma,1}(k,\Z)$ for any $m\leq 3$ or $m\geq 6$. For the rest of the cases we have the following l.e.s.
  $$0\rightarrow H^{4-4\sigma,1} _{C _2}(k,\Z)\rightarrow H^{4-3\sigma,1} _{C _2}(k,\Z)=k^*/k^{*2}\stackrel{\alpha _1}{\rightarrow} H^{1,1}(k,\Z)=k^*\rightarrow H^{5-4\sigma,1} _{C _2}(k,\Z)\rightarrow 0.$$
The composition $S^{\sigma}\rightarrow C _{2+}\wedge S^1\rightarrow S^{\sigma}$ gives multiplication by 2 (i.e. zero map) on the composition
  $$H^{4-3\sigma,1} _{C _2}(k,\Z)=k^*/k^{*2}\stackrel{\alpha _1}{\rightarrow} H^{1,1}(k,\Z)=k^*\rightarrow H^{4-3\sigma,1} _{C _2}(k,\Z)=k^*/k^{*2}$$
  from  Theorem \ref{comp} and the second map is surjective because $H^{1,1}(k,\Z)=k^*\rightarrow H^{4-3\sigma,1} _{C _2}(k,\Z)=k^*/k^{*2}\rightarrow H^{4-2\sigma} _{C _2}(k,\Z)=0$. The  map $\alpha _1$ is zero when $k$ is a quadratically closed field. Suppose $k$ is a formally real field.  Then $\alpha^2 _1(x)=1$ which implies  $Im(\alpha _1)\subset \Z/2$. Because  $Im(\alpha _1)\subset k^{*2}$ and $k$ is a formally real field it implies that $\alpha _1$ is the zero map. 
  
  We conclude that $H^{4-4\sigma,1} _{C _2}(k,\Z)=k^*/k^{*2}$,  $H^{3-4\sigma,1} _{C _2}(k,\Z)=H^{3-3\sigma,1} _{C _2}(k,\Z)=\Z/2$, $H^{5-4\sigma,1} _{C _2}(k,\Z)=k^*$ and $H^{n-4\sigma,1} _{C _2}(k,\Z)=0$ for $n\leq 2$ and $n\geq 6$. 

If $n=5$ we have the following sequence 
$$0\rightarrow H^{5-5\sigma,1} _{C _2}(k,\Z)\rightarrow H^{5-4\sigma,1} _{C _2}(k,\Z)=k^*\stackrel{\alpha _2}{\rightarrow} H^{1,1}(k,\Z)=k^*\rightarrow H^{6-5\sigma,1} _{C _2}(k,\Z)\rightarrow 0.$$
From the above we have that $$0\rightarrow H^{1,1}(k,\Z)=k^*\simeq H^{5-4\sigma,1} _{C _2}(k,\Z)=k^*\rightarrow H^{5-3\sigma,1} _{C _2}(k,\Z)=0$$ so by Theorem \ref{comp} we obtain that $\alpha _2$ is multiplication by 2. We conclude that $H^{5-5\sigma,1} _{C _2}(k,\Z)=\Z/2$ and $H^{6-5\sigma,1} _{C _2}(k,\Z)=k^*/k^{*2}$. We also have 
$H^{n-5\sigma,1} _{C _2}(k,\Z)=H^{n-4\sigma,1} _{C _2}(k,\Z)$ for $n\neq 5,6$.
 
By induction  we have the following:

 If $n>2$ is even then $H^{n+1-n\sigma,1} _{C _2}(k,\Z)=k^*$ and $H^{m-n\sigma,1} _{C _2}(k,\Z)=k^*/k^{*2}$ for $2<m\leq n$ even and  $H^{m-n\sigma,1} _{C _2}(k,\Z)=\Z/2$ with $2<m<n$, $m$ odd. 
 
 If $n>2$ is odd then $H^{m-n\sigma,1} _{C _2}(k,\Z)=k^*/k^{*2}$ for $2<m\leq n+1$, $m$ even and $H^{m-n\sigma,1} _{C _2}(k,\Z)=\Z/2$ if $2<m\leq n+1$, $m$ odd.
 
 The case $n=0$ follows from the computation of motivic cohomology of a field in weight 1.
   \end{proof}
    \begin{corollary}\label{fcoe} Let $k$ be a field of characteristic zero. If $k$ is a euclidian field or a quadratically closed field we have that the positive cone ($n\geq 0$) is
 $$
H^{n\sigma-m,1}_{C _2}(k, \Z/2) =
\begin{cases}
k^*/k^{*2}\oplus\Z/2 & 0\leq m\leq n-1\\
\Z/2&  0\leq m=n\\
k^*/k^{*2} & m=-1<n\\
0& otherwise\\ 
\end{cases}.
$$
and the negative cone ($-n<0$)
$$
H^{m-n\sigma,1}_{C _2}(k, \Z/2) =
\begin{cases}
k^*/k^{*2}\oplus\Z/2 & 3\leq m\leq n\\
\Z/2&  m=2\leq n\\
k^*/k^{*2} & m=n+1, n\geq 2\\
0& otherwise\\ 
\end{cases}.
$$
 \end{corollary} 
 We conclude that for a quadratically closed field the Bredon motivic cohomology of weight 1 with $\Z/2-$coefficients coincides as abstract groups with Bredon cohomology of a point. For a euclidian field like $\R$ we have that $k^*/k^{*2}=\Z/2$.

The negative cone of Bredon motivic cohomology of weight $\sigma$ of a formally real field or a quadratically closed field is computed below. For a quadratically closed field it coincides with the negative cone in weight 1 (see Theorem \ref{neg1}). It doesn't coincide for a formally real field.
\begin{theorem} \label {sigm}
 Suppose that $k$ is a quadratically closed or a formally real field  ($char(k)=0$). Then for $n\geq 1$ we have:

If $n\geq 2$ is even  then
 $$
H^{m-n\sigma,\sigma}_{C _2}(k, \Z) =
\begin{cases}
k^*& m=n+1\\
k^*/k^{*2}& 2\leq m\leq n, m=\text{even}\\
\Z/2 & 2<m\leq n, m=\text{odd} \\
0& otherwise\\ 
\end{cases}.
$$
If $n\geq 1$ is odd then 
$$
H^{m-n\sigma,\sigma}_{C _2}(k, \Z) =
\begin{cases}
k^*/k^{*2}& 2\leq m\leq n+1, m=\text{even}\\
\Z/2 & 2<m\leq n, m=\text{odd} \\
0& otherwise\\ 
\end{cases}.
$$
\end{theorem}
\begin{proof}We have that $H^{n\sigma+m,\sigma}_{C _2}(k,\Z)=H^{2\sigma+m,\sigma}(k,\Z)=0$ from l.e.s. for any $n\geq 2$ and $m\geq 0$. We also have 
$$H^{-m,1}_{C _2}(k,\Z)\rightarrow H^{-m\sigma,\sigma}_{C _2}(k,\Z)\rightarrow H^{(-m+1)\sigma,\sigma}_{C _2}(k,\Z)\rightarrow H^{-m+1,1}_{C _2}(k,\Z).$$
We have $H^{-m\sigma,\sigma}_{C _2}(k,\Z)=0=H^{0,\sigma}_{C _2}(k,\Z)=0$ for any $m\geq 0$. Also $H^{n-m\sigma,\sigma}_{C _2}(k,\Z)=0$ for any $n\leq 0$ and $m\geq 0$. 
We have that 
$$H^{0,1}_{C _2}(k,\Z)=0\rightarrow H^{1-\sigma,\sigma}_{C _2}(k,\Z)\rightarrow H^{1,\sigma}_{C _2}(k,\Z)=k^{*2}\stackrel{\beta _1}\rightarrow $$
$$\rightarrow H^{1,1}_{C _2}(k,\Z)=k^*\rightarrow H^{2-\sigma,\sigma}_{C _2}(k,\Z)\rightarrow H^{2,\sigma}_{C _2}(k,\Z)=0.$$
The map $\beta _1$ is an injective map because $H^{1-\sigma,\sigma}_{C _2}(k,\Z)=0$. This follows from $H^{1-\sigma,\sigma}_{C _2}(k,\Z)=H^{1-\sigma,\sigma} _{C _2}(\EG \Z/2,\Z)=H^{3-3\sigma,1} _{C _2}(\EG \Z/2, \Z)$ and the motivic isotropy sequence gives
$$H^{2-3\sigma,1} _{C _2}(\EG\Z/2,\Z)=H^{0,\sigma} _{C _2}(k,\Z)=0\rightarrow \Z/2\rightarrow H^{3-3\sigma,1} _{C _2}(k,\Z)=\Z/2\rightarrow H^{3-3\sigma,1} _{C _2}(\EG \Z/2,\Z)\rightarrow 0.$$
It implies that $H^{3-3\sigma,1} _{C _2}(\EG \Z/2,\Z)=0$. 

In conclusion $\beta _1$ is injective and $H^{2-\sigma,\sigma}_{C _2}(k,\Z)=k^*/k^{*2}$. We conclude that $H^{n-\sigma,\sigma}_{C _2}(k,\Z)=k^*/k^{*2}$ for $n=2$ and zero otherwise. 
We also notice that $$H^{1-\sigma,\sigma}_{C _2}(k,\Z)\simeq H^{1-n\sigma,\sigma}_{C _2}(k,\Z)=0$$ for any $n\geq 1$.

We have
$$0\rightarrow H^{2-2\sigma,\sigma}_{C _2}(k,\Z)\rightarrow H^{2-\sigma,\sigma}_{C _2}(k,\Z)=k^*/k^{*2}\stackrel{\tau}{\rightarrow}H^{1,1}_{C _2}(k,\Z)=k^*\rightarrow H^{3-2\sigma,\sigma}_{C _2}(k,\Z)\rightarrow  H^{3-\sigma,\sigma}_{C _2}(k,\Z)=0$$
so $H^{2-2\sigma,\sigma}_{C _2}(k,\Z)=k^*/k^{*2}$ and $H^{3-2\sigma,\sigma}_{C _2}(k,\Z)=k^*$ because the map $\tau$ is the zero map because $k$ is either a formally real field or a quadratically closed field.

We also have $H^{n-2\sigma,\sigma}_{C _2}(k,\Z)=0$ for any $n\neq 2, 3$.  

We have 
$$0\rightarrow H^{3-3\sigma,\sigma}_{C _2}(k,\Z)\rightarrow H^{3-2\sigma,\sigma}_{C _2}(k,\Z)=k^*\stackrel{\tau _1}{\rightarrow} H^{1,1}_{C _2}(k,\Z)=k^*\rightarrow H^{4-3\sigma,\sigma}_{C _2}(k,\Z)\rightarrow 0$$
and the map $\tau _1$ is the multiplication by 2 as above so $H^{3-3\sigma,\sigma}_{C _2}(k,\Z)=\Z/2$ and $H^{4-3\sigma,\sigma}_{C _2}(k,\Z)=k^*/k^{*2}$. Also $H^{2-3\sigma,\sigma}_{C _2}(k,\Z)=H^{2-2\sigma,\sigma}_{C _2}(k,\Z)=k^*/k^{*2}$ and $H^{n-3\sigma,\sigma}_{C _2}(k,\Z)=0$ if $n\neq 2,3,4$. 

For $n=4$ we have 
$$0\rightarrow H^{4-4\sigma,\sigma}_{C _2}(k,\Z)\rightarrow H^{4-3\sigma,\sigma}_{C _2}(k,\Z)=k^*/k^{*2}\stackrel{\tau _2}{\rightarrow} H^{1,1}_{C _2}(k,\Z)=k^*\rightarrow H^{5-4\sigma,\sigma}_{C _2}(k,\Z)\rightarrow H^{5-3\sigma}_{C _2}(k,\Z)=0$$
so $H^{4-4\sigma,\sigma}_{C _2}(k,\Z)=k^*/k^{*2}$ and $H^{5-4\sigma,\sigma}_{C _2}(k,\Z)=k^*$. This is because the map $\tau _2$ is the zero map as above. Also $H^{3-4\sigma,\sigma}_{C _2}(k,\Z)=H^{3-3\sigma,\sigma}_{C _2}(k,\Z)=\Z/2$, $H^{2-4\sigma,\sigma}_{C _2}(k,\Z)=H^{2-3\sigma,\sigma}_{C _2}(k,\Z)=k^*/k^{*2}$ and the rest are zero. 

By induction we have that if $n$ is even then $H^{2-n\sigma,\sigma}_{C _2}(k,\Z)=k^*/k^{*2}$, $H^{3-n\sigma,\sigma}_{C _2}(k,\Z)=\Z/2$,..., $H^{(n-1)-n\sigma,\sigma}_{C _2}(k,\Z)=\Z/2$, $H^{n-n\sigma,\sigma}_{C _2}(k,\Z)=k^*/k^{*2}$, $H^{n+1-n\sigma,\sigma}_{C _2}(k,\Z)=k^*$.
If $n$ odd then $H^{2-n\sigma,\sigma}_{C _2}(k,\Z)=k^*/k^{*2}$, $H^{3-n\sigma,\sigma}_{C _2}(k,\Z)=\Z/2$,..., $H^{n-1-n\sigma,\sigma}_{C _2}(k,\Z)=k^*/k^{*2}$, $H^{n-n\sigma,\sigma}_{C _2}(k,\Z)=\Z/2$, $H^{n+1-n\sigma,\sigma}_{C _2}(k,\Z)=k^*/k^{*2}$.
\end{proof}
In the following theorem we compute the positive cone of Bredon motivic cohomology of weight $\sigma$ of a quadratically closed field and a euclidian field.
\begin{theorem} \label{pos1} Let $k$ be a quadratically closed field or a euclidian field. Then if $n\geq 2$ is even 
$$
H^{n\sigma-m,\sigma}_{C _2}(k, \Z) =
\begin{cases}
k^*& m=n-1\\
k^*/k^{*2}& 1\leq m<n-1, m=\text{odd}\\
\Z/2 & 0\leq m<n-1, m=\text{even} \\
0& otherwise.\\ 
\end{cases}.
$$
If $n\geq 1$ is odd then 
$$
H^{n\sigma-m,\sigma}_{C _2}(k, \Z) =
\begin{cases}
k^*/k^{*2}& 1\leq m<n-1, m=\text{odd}\\
\Z/2 & 0\leq m\leq n-1, m=\text{even} \\
0& otherwise\\ 
\end{cases}.
$$
If $n=0$ then $H^{1,\sigma}(k,\Z)=k^{*2}$ and $H^{m,\sigma}(k,\Z)=0$ for any $m\neq 1$.
\end{theorem}
\begin{proof}
 From Proposition \ref{Nie} we have that $H^{2\sigma-1,\sigma}_{C _2}(k,\Z)=k^*$, $H^{2\sigma,\sigma}_{C _2}(k,\Z)=\Z/2$ and $H^{2\sigma+m,\sigma}_{C _2}(k,\Z)=0$ for $m\neq -1,0$. We have 
 $$0\rightarrow H^{2\sigma-2,\sigma}_{C _2}(k,\Z)=0\rightarrow H^{3\sigma-2,\sigma}_{C _2}(k,\Z)\rightarrow $$
 $$\rightarrow H^{1,1}_{C _2}(k,\Z)=k^*\stackrel{\gamma}\rightarrow H^{2\sigma-1,\sigma}_{C _2}(k,\Z)=k^*\rightarrow H^{3\sigma-1,\sigma}_{C _2}(k,\Z)\rightarrow 0$$ 
 with the map $\gamma$ is given by multiplication by $2$. This follows from the fact that $H^{2\sigma-1,\sigma}_{C _2}(k,\Z)\rightarrow H^{1,1}_{C _2}(k,\Z)$ is the identity map (Proposition \ref{sigma}) and from the following diagram 
 $$
\xymatrix{
\rH^{3\sigma-1,\sigma} _{C _2}(S^\sigma)\ar[r]^\iso\ar[d] _{\iso}& \rH^{3\sigma-1, \sigma}_{C_2}(C_{2\,+}\wedge S^1) \ar[r]^{\gamma}\ar[d] & \rH^{3\sigma-1, \sigma}_{C_{2}}(S^\sigma) \ar[d]\ar[d]_{\iso} \\
\rH^{2,1}_{\mcal{M}}(S^1) \ar[r] & \rH^{2,1}_{\mcal{M}}(S^1\vee S^1) \ar[r] &
\rH^{2,1}_{\mcal{M}}(S^1).
} 
$$

 Then $H^{3\sigma-2,\sigma}_{C _2}(k,\Z)=\Z/2$ and $H^{3\sigma-1,\sigma}_{C _2}(k,\Z)=k^*/k^{*2}$. We also have $H^{3\sigma,\sigma}_{C _2}(k,\Z)=H^{2\sigma,\sigma}_{C _2}(k,\Z)=\Z/2$. For $n=4$ we have 
 $$0\rightarrow H^{3\sigma-3,\sigma}_{C _2}(k,\Z)=0\rightarrow H^{4\sigma-3,\sigma}_{C _2}(k,\Z)\rightarrow $$
 $$\rightarrow H^{1,1}_{C _2}(k,\Z)=k^*\stackrel{\alpha}{\rightarrow} H^{3\sigma-2,\sigma}_{C _2}(k,\Z)=\Z/2\rightarrow H^{4\sigma-2,\sigma}_{C _2}(k,\Z)\rightarrow 0.$$
 The map $\alpha$ is zero for a quadratically closed field or a euclidian field because  $H^{2\sigma-2,\sigma}_{C _2}(k,\Z)=0\rightarrow H^{3\sigma-2,\sigma}_{C _2}(k,\Z)=\Z/2\rightarrow H^{1,1}_{C _2}(k,\Z)$ is injective and we have the following diagram 
$$
\xymatrix{
\rH^{4\sigma-2,1} _{C _2}(S^\sigma)\ar@{^{(}->}[r]\ar@{^{(}->}[d]& \rH^{4\sigma-2, 1}_{C_2}(C_{2\,+}\wedge S^1) \ar[r]^{\alpha}\ar[d] & \rH^{4\sigma-2, 1}_{C_{2}}(S^\sigma) \ar[d]\ar@{^{(}->}[d] \\
\rH^{2,1}_{\mcal{M}}(S^1) \ar[r] & \rH^{2,1}_{\mcal{M}}(S^1\vee S^1) \ar[r] &
\rH^{2,1}_{\mcal{M}}(S^1).
} 
$$

This implies that $H^{4\sigma-3,\sigma}_{C _2}(k,\Z)=k^*$ and $H^{4\sigma-2,\sigma}_{C _2}(k,\Z)=\Z/2$. Also $H^{4\sigma,\sigma}_{C _2}(k,\Z)=H^{2\sigma,\sigma}_{C _2}(k,\Z)=\Z/2$,  $H^{4\sigma-1,\sigma}_{C _2}(k,\Z)=H^{3\sigma-1,\sigma}_{C _2}(k,\Z)=k^*/k^{*2}$. For $n=5$ we have 
$$0\rightarrow H^{4\sigma-4,\sigma}_{C _2}(k,\Z)=0\rightarrow H^{5\sigma-4,\sigma}_{C _2}(k,\Z)\rightarrow $$
$$\rightarrow H^{1,1}_{C _2}(k,\Z)=k^*\stackrel{\beta}\rightarrow H^{4\sigma-3,\sigma}_{C _2}(k,\Z)=k^*\rightarrow H^{5\sigma-3,\sigma}_{C _2}(k,\Z)\rightarrow 0$$
with the map $\beta$ multiplication by $2$ because $H^{3\sigma-3,\sigma}_{C _2}(k,\Z)=0\rightarrow H^{4\sigma-3,\sigma} _{C _2}(k,\Z)\simeq H^{1,1}(k,\Z)$ and we have the following diagram
$$
\xymatrix{
\rH^{5\sigma-3,\sigma} _{C _2}(S^\sigma)\ar[r]^\iso\ar[d] _{\iso}& \rH^{5\sigma-3, \sigma}_{C_2}(C_{2\,+}\wedge S^1) \ar[r]^{\beta}\ar[d] & \rH^{5\sigma-3, \sigma}_{C_{2}}(S^\sigma) \ar[d]\ar[d]_{\iso} \\
\rH^{2,1}_{\mcal{M}}(S^1) \ar[r] & \rH^{2,1}_{\mcal{M}}(S^1\vee S^1) \ar[r] &
\rH^{2,1}_{\mcal{M}}(S^1).
} 
$$
with the lower horizontal composition being multiplication by 2. It implies that  $H^{5\sigma-4,\sigma}_{C _2}(k,\Z)=\Z/2$ and $H^{5\sigma-3,\sigma}_{C _2}(k,\Z)=k^*/k^{*2}$. By previous computations we obtain $H^{5\sigma-2,\sigma}_{C _2}(k,\Z)=\Z/2$, $H^{5\sigma-1,\sigma}_{C _2}(k,\Z)=k^*/k^{*2}$, $H^{5\sigma,\sigma}_{C _2}(k,\Z)=\Z/2$. 

By induction it implies that if $n$ is even then $H^{n\sigma,\sigma}_{C _2}(k,\Z)=\Z/2$, $H^{n\sigma-1,\sigma}_{C _2}(k,\Z)=k^*/k^{*2}$,
..., $H^{n\sigma-n+2,\sigma}_{C _2}(k,\Z)=\Z/2$, $H^{n\sigma-n+1,\sigma}_{C _2}(k,\Z)=k^*$. If $n$ is odd then $H^{n\sigma,\sigma}_{C _2}(k,\Z)=\Z/2$, $H^{n\sigma-1,\sigma}_{C _2}(k,\Z)=k^*/k^{*2}$,
..., $H^{n\sigma-n+2,\sigma}_{C _2}(k,\Z)=k^*/k^{*2}$, $H^{n\sigma-n+1,\sigma}_{C _2}(k,\Z)=\Z/2$.

The case $n=0$ was discussed in Proposition \ref{sigma}.
\end{proof}

\begin{corollary}\label{sif} Let $k$ be a field of characteristic zero and $n>0$. If $k$ is a euclidian field or a quadratically closed field then the negative cone is
$$
H^{m-n\sigma,\sigma}_{C _2}(k, \Z/2) =
\begin{cases}
k^*/k^{*2}\times\Z/2 & 2\leq m\leq n \\
k^*/k^{*2}& 2\leq  m=n+1, m=1\\
0& otherwise\\ 
\end{cases}.
$$
and the positive cone is 
$$
H^{n\sigma-m,\sigma}_{C _2}(k, \Z/2) =
\begin{cases}
k^*/k^{*2}\times\Z/2 & 1\leq m\leq n-1 \\
\Z/2 & 1\leq m=n, m=0\\
0& otherwise\\ 
\end{cases}.
$$
If $n=0$ we have that $$
H^{0,\sigma}_{C _2}(k, \Z/2) =
\begin{cases}
\Z/2 & -1\in k^{*2} \\
0 & -1\notin k^{*2}\\ 
\end{cases}.
$$
and $H^{1,\sigma}(k,\Z/2)=k^{*2}/k^{*4}$ and $H^{n,\sigma}(k,\Z/2)=0$ for $n\neq 0,1$.
\end{corollary}
\begin{proof} Universal coefficients sequence gives a split s.e.s
$$0\rightarrow H^{n\sigma-m,\sigma} _{C _2}(k,\Z)\otimes \Z/2\rightarrow H^{n\sigma-m,\sigma} _{C _2}(k,\Z/2)\rightarrow  {_2H}^{n\sigma-m+1,\sigma} _{C _2}(k,\Z)\rightarrow 0.$$
We have $H^{n\sigma-m,\sigma} _{C _2}(k,\Z/2)\simeq H^{n\sigma-m,\sigma} _{C _2}(k,\Z)\otimes \Z/2\oplus {_2H}^{n\sigma-m+1,\sigma} _{C _2}(k,\Z)$. 

Suppose $n$ is even. If $1\leq m< n-1$ and $m$ odd then $0\leq m-1<n-2$ with $m-1$ even. Then according to Theorem \ref{pos1} we have $H^{n\sigma-m,\sigma} _{C _2}(k,\Z/2)\simeq k^*/k^{*2}\oplus \Z/2$. In the case $m=n-1$ we have that $H^{n\sigma-m,\sigma} _{C _2}(k,\Z/2)\simeq k^*\otimes\Z/2\oplus {\Z/2}\simeq k^*/k^{*2}\oplus \Z/2$.

If $2\leq m< n-1$ and $m$ even then $1\leq m-1\leq n-2$ with $m-1$ odd. Then $H^{n\sigma-m,\sigma} _{C _2}(k,\Z/2)\simeq \Z/2\oplus  k^*/k^{*2}$. 
If $1\leq m=n$ then $H^{n\sigma-m,\sigma} _{C _2}(k,\Z/2)\simeq  {_2}k^*\simeq \Z/2$ and if $m=0$ then $H^{n\sigma-m,\sigma} _{C _2}(k,\Z/2)\simeq \Z/2$.

Suppose $n$ is odd. If $1\leq m< n-1$ and $m$ odd then $0\leq m-1<n-2$ with $m-1$ even. The decomposition is in this case  $H^{n\sigma-m,\sigma} _{C _2}(k,\Z/2)\simeq k^*/k^{*2}\oplus \Z/2$. If $2\leq m< n-1$ and $m$ even then $1\leq m-1\leq n-2$ with $m-1$ odd. The decomposition is in this case $H^{n\sigma-m,\sigma} _{C _2}(k,\Z/2)\simeq \Z/2\oplus k^*/k^{*2}$. If $1\leq m=n$ then the decomposition is $H^{n\sigma-m,\sigma} _{C _2}(k,\Z/2)\simeq \Z/2$ and if $m=0$ the decomposition is $H^{n\sigma-m,\sigma} _{C _2}(k,\Z/2)\simeq \Z/2$. 

For the negative cone we have $$H^{m-n\sigma,\sigma} _{C _2}(k,\Z/2)\simeq H^{m-n\sigma,\sigma} _{C _2}(k,\Z)\otimes \Z/2\oplus {_2H}^{m+1-n\sigma,\sigma} _{C _2}(k,\Z).$$ The result follows from Theorem \ref{sigm}. If $n\geq 2$ is even then we have the following cases. If $2\leq m\leq n-1$, $m$ even then $2<m+1\leq n$, $m+1$ odd and the decomposition is $H^{m-n\sigma,\sigma} _{C _2}(k,\Z/2)\simeq k^*/k^{*2}\oplus \Z/2$. If $2<m<n$, $m$ odd then $2<m+1<n+1$, $m+1$ even and the decomposition is  $H^{m-n\sigma,\sigma} _{C _2}(k,\Z/2)\simeq \Z/2\oplus k^*/k^{*2}$. If $1\leq m=n+1$ then the decomposition is  $H^{m-n\sigma,\sigma} _{C _2}(k,\Z/2)\simeq k^*/k^{*2}$. If $m=1$ then the decomposition is $H^{m-n\sigma,\sigma} _{C _2}(k,\Z/2)\simeq k^*/k^{*2}$.

If $n\geq 2$ is odd then we have the following cases. If $2\leq m\leq n$ and $m$ is even then the decomposition is $H^{m-n\sigma,\sigma} _{C _2}(k,\Z/2)=k^*/k^{*2}\oplus \Z/2$. If $2<m\leq n$, $m$ odd then the decomposition is $H^{m-n\sigma,\sigma} _{C _2}(k,\Z/2)=\Z/2\oplus k^*/k^{*2}$. If $m=1$ then $H^{m-n\sigma,\sigma} _{C _2}(k,\Z/2)=k^*/k^{*2}$
and if $1\leq m=n+1$ then $H^{m-n\sigma,\sigma} _{C _2}(k,\Z/2)=k^*/k^{*2}$.

\end{proof}
Notice that for a quadratically closed field the Bredon motivic cohomology groups $H^{a+p\sigma,1} _{C _2}(k,\Z/2)$ and $H^{a+p\sigma,\sigma} _{C _2}(k,\Z/2)$ are computed as Bredon cohomology  group $H^{a,p} _{Br}(pt,\Z/2)$. The positive cones of Bredon motivic cohomology of weight $1$ and $\sigma$ of a quadratically closed field coincide (this is true also for the negative cones). 

\section{Borel equivariant motivic cohomology in weights $0,1,\sigma$}
 In \cite{HOV1} we extended Edidin-Graham equivariant higher Chow groups to an equivariant theory represented in the stable equivariant $\A^1-$homotopy category and bigraded by virtual $C _2-$representations. In Proposition 4.1 \cite{HOV1} we showed that $$CH^b _{C _2}(X,2b-a,A)\simeq H^{a,b} _{C _2}(X\times \EG C _2,A)$$ for any integers $a,b$ and abelian group $A$. 
 
 It implies that Borel equivariant motivic cohomology groups  $H^{a+p\sigma,b+q\sigma}_{C _2}(X\times \EG C _2,A)$ are an extension of Edidin-Graham equivariant higher Chow groups introduced in \cite{EG}. For weights $0,1,\sigma$  and an abelian group $A$ we have the following computation of the Borel equivariant motivic cohomology of a field:
 \begin{proposition} \label{borel} a) $H^{a+p\sigma,0} _{C _2}(\EG C _2,A)=H^{a+p\sigma,0}_{C _2}(k,A)$.
 
 b) $H^{a+p\sigma,\sigma}_{C _2}(\EG C _2,A)=H^{a+p\sigma,\sigma}_{C _2}(k,A)$.
 
 c) $H^{a+p\sigma,1}_{C _2}(\EG C _2,A)=H^{a+p\sigma,1}_{C _2}(k,A)$ if $a\neq 2,3$.
 
 d) $H^{a+p\sigma,1}_{C _2}(\EG C _2,A)=H^{a-2+(p+2)\sigma,\sigma}_{C _2}(\EG C _2,A)=H^{a-2+(p+2)\sigma,\sigma}_{C _2}(k,A)$ if $a=2,3$. 
 \end{proposition}

\begin{proof} For a) we have that the equivariant motivic isotropy sequence $\EG C_2\rightarrow S^0\rightarrow \EGt C _2$  gives the following l.e.s. 
$$\rH^{a+p\sigma,0}_{C_2}(\EGt C _2,A)=0\rightarrow H^{a+p\sigma,0}_{C_2}(k,A)\rightarrow H^{a+p\sigma,0}_{C_2}(\EG C _2)\rightarrow \rH^{a+1+p\sigma,0}_{C_2}(\EGt C_2)=0.$$
The same motivic isotropy sequence gives 
$$H^{2+p\sigma,1}_{C_2}(\EG C_2,A)\rightarrow \rH^{3+p\sigma,1}_{C_2}(\EGt C _2,A)\rightarrow H^{3+p\sigma,1}_{C_2}(k,A)\rightarrow H^{3+p\sigma,1}_{C_2}(\EG C_2).$$
with  $\rH^{3+p\sigma,1}_{C_2}(\EGt C _2,A)=\Z/2$ and  $\rH^{n+p\sigma,1}_{C_2}(\EGt C _2,A)=0$ for $n\neq 3$ (Proposition \ref{case1}).

We have from the same isotropy sequence that $H^{a+p\sigma,\sigma}_{C _2}(\EG C _2,A)=H^{a+p\sigma,\sigma}_{C _2}(k,A)$ because of the periodicities of Bredon motivic cohomology of $\EGt C _2$. This implies part b).

We also have a $(2\sigma-2,\sigma-1)$ periodicity for Bredon motivic cohomology groups of $\EG C _2$. This implies part d). 
\end{proof}
Edidin-Graham higher Chow groups of $\EG C_2$ in weight 0 and 1 are $CH^1 _{C _2}(\EG C_2,0,\Z)=CH^1_{C_2}(\BG C _2)=\Z/2= H^{2\sigma,\sigma} _{C _2}(k,\Z)$ (from d) above), 
$CH^1 _{C _2}(\EG C_2,1,\Z)=k^*=H^{1,1} _{C _2}(k,\Z)$ (from c) above) and $CH^0 _{C _2}(\EG C_2,0,\Z)=\Z=H^{0,0}_{C _2}(k,\Z)$, $CH^0 _{C _2}(\EG C_2,-a,\Z)=0=H^{a,0}_{C _2}(k,\Z)$, for $a\neq 0$ (from a) above).

 Notice that the canonical map $H^{a+p\sigma,b+q\sigma}(k,\Z)\rightarrow H^{a+p\sigma,b+q\sigma}(EC _2,\Z)$ is not necessarily an isomorphism (for example for $a=2$, $b=1$, $q=0$, $p=0$ the map is $0\rightarrow \Z/2$).

 \bibliographystyle{plain}
 \bibliography{road}
 \end{document}